\documentclass[onefignum,onetabnum]{siamart171218}

\usepackage[T1]{fontenc}
\usepackage[utf8]{inputenc}
\usepackage[english]{babel}
\usepackage{algorithm}
\usepackage{algorithmic}

\usepackage{amsmath}
\usepackage{amsfonts}
\usepackage{amssymb}
\usepackage{mathtools}

\usepackage{accents}
\usepackage{bm}
\usepackage{xcolor}

\usepackage{graphicx}
\usepackage[section]{placeins}
\usepackage{booktabs}
\usepackage{xcolor,colortbl}
\usepackage{color}
\newcommand\munderbar[1]{\underaccent{\bar}{#1}}

\headers{Three-field optimization formulation for flow in DFN}{Berrone, Grappein, Pieraccini, Scial\`o}

\title{A three-field based optimization formulation for flow simulations in networks of fractures on non conforming meshes
\thanks{This work was supported by the MIUR project ``Dipartimenti di Eccellenza 2018-2022'' (CUP E11G18000350001), PRIN project "Virtual Element Methods: Analysis and Applications" (201744KLJL\_004) and by INdAM-GNCS.}
}

\author{Stefano Berrone\thanks{Dipartimento di Scienze Matematiche, Politecnico di Torino, Torino, IT. Member of INdAM-GNCS group. \email{stefano.berrone@polito.it,  denise.grappein@polito.it, stefano.scialo@polito.it}}
\and
Denise Grappein \footnotemark[2]
\and
Sandra Pieraccini\thanks{Dipartimento di Ingegneria Meccanica e Aerospaziale, Politecnico di Torino, Torino, IT. Member of INdAM-GNCS group. \email{sandra.pieraccini@polito.it}} 
\and
Stefano Scial\`o \footnotemark[2]}

\begin{document}

\maketitle

\begin{center}
\today
\end{center}

\begin{abstract}
A new numerical scheme is proposed for flow computation in complex discrete fracture networks. The method is based on a three-field formulation of the Darcy law for the description of the hydraulic head on the fractures and uses a cost functional to enforce the required coupling condition at fracture intersections. The resulting method can handle non conforming meshes, independently built on each geometrical object of the computational domain, and ensures local mass conservation properties at fracture intersections. An iterative solver is devised for the method, ready for parallel implementation on parallel computing architectures.
\end{abstract}



\section{Introduction}
The present work proposes a new numerical approach for flow simulations in fracture networks, described by means of the Discrete Fracture Network (DFN) model. DFNs are sets of intersecting planar polygons arbitrarily oriented in the three dimensional space, representing fractures in underground rock formations, and are typically generated starting from probability distribution functions on hydraulic and geological soil properties \cite{CLMTBDFP,DeDreuzyEtAlii01,Dowd2007}. DFN models, since providing an explicit representation of fractures, are a viable alternative to homogenization based approaches \cite{Berkowitz2002}, when the presence of a network of fractures sensibly affects relevant flow characteristics. In fact flow directionality and preferential paths might not be correctly accounted for by using homogenized properties for rocks and fractures \cite{Fidelibus2009,Sahimi2011}. When fracture hydraulic transmissivity is much higher than rock transmissivity, the influence of the porous rock matrix can be neglected, with minor impact on the prediction of the flow. 
%

A major drawback for DFN flow simulations is related to the geometrical complexity and size of the resulting computational domains, which might count a large number of fractures, with dimensions ranging from centimeters to kilometers and forming an intricate network of intersections, where suitable conditions need to be enforced to couple the solution on the intersecting fractures. This complex multi-scale geometrical nature of DFN domains significantly limits the applicability of conventional numerical simulation tools which rely on mesh conformity to enforce interface conditions, as it is often a very difficult task to generate good quality conforming meshes of realistic DFNs, even introducing a large number of unknowns, \cite{deDreuzy13,Antonietti2016,FKS}.

Recently, many different approaches have been suggested to overcome such a difficulty. A possible strategy consists in a dimensional reduction of the problem: in \cite{DF99,BODIN2007} the DFN is replaced by a set of one-dimensional channels or pipes resembling the connections among fractures in the network; in \cite{NOETJCP15} the problems on the fractures are re-written in terms of the 1D interface unknowns only, whereas in \cite{Karra2018,HOBE2018,HymanMLNature} DFNs are analyzed using graph theory tools. 
Some authors propose new efficient meshing strategies for complex networks, aiming at obtaining a conforming mesh with minor modifications of network geometry \cite{Fourno2019}, or replacing hard-to-mesh configurations with stochastically equivalent analogues, which are easier to mesh \cite{Gable2014,Gable2015bTransp}. Discretization methods capable of handling polygonal meshes are also suggested as effective strategies to obtain conforming meshes of complex networks: the use of Virtual Elements is proposed in \cite{BBBPS,BBSmixed,Fumagalli2018,BBD}, Mimetic Finite Differences in \cite{Antonietti2016}  and Hybrid High-Order Methods in \cite{Chave2018}, as some relevant examples. Other authors suggest the use of mortaring techniques to partially alleviate the conformity requirement at fracture intersections, \cite{Vohralik2007,Pichot2012,Pichot2014}.

The present work takes inspiration from a different approach, proposed in \cite{BPSa,BPSd,BPSe}, which relies on numerical optimization to enforce interface conditions, without requiring any mesh conformity at fracture intersections, and thus completely overcoming any problem related to mesh generation. A cost functional, expressing the error in fulfilling interface conditions, is minimized constrained by a set of partial differential equations written on each fracture. The method is robust to complex geometries and highly efficient thanks to its predisposition to parallel implementation \cite{BBV2019,BSV}. Here, while keeping a similar optimization framework, a new formulation is proposed for the constraint equations, based on the \textit{three-field} formulation suggested in \cite{Brezzi}. The resulting approach retains the capability of dealing with non-conforming meshes and the predisposition to parallel implementation given by the optimization formulation. It is now based on a novel mono-objective functional definition and has intrinsic properties of local mass-conservation across traces.

The manuscript is organized as follows: in Section~\ref{ContModel} a three-field formulation of the Darcy problem in fracture networks is written and recast into a PDE-constrained formulation suitable for discretization on non conforming meshes. The resulting discrete approach is shown in Section~\ref{DiscreteModel}. Section~\ref{DiscrWellPosedness} reports well posedness results for the discrete problem, Section~\ref{prob_solve} the algorithm proposed to compute the numerical solution and Section~\ref{NumRes} describes some numerical examples. Concluding remarks are finally proposed in Section~\ref{conclusions}.

\section{Continuous model}
\label{ContModel}
\newcommand{\Lrm}{\mathrm{L}}
\newcommand{\Hrm}{\mathrm{H}}
The present Section is devoted to the presentation of a three-field formulation for the Darcy problem: in the first subsection, a classical formulation is proposed, introducing the equations and the coupling conditions at the interfaces, whereas, in the second subsection a novel optimization formulation is described.
In the following, $\Lrm^2(\omega)$ is the Hilbert space on $\omega$ of square integrable functions, and $\Hrm^1(\omega)$ refers to the classical Sobolev space of order one on $\omega$; inner products in a function space $V$ are denoted by $\left(\cdot,\cdot\right)_{V}$, whereas $\left\langle\cdot,\cdot\right\rangle_{V,V'}$ is a duality pairing between spaces $V$ and $V'$. Notation $v_{|_\gamma}$ denotes the trace on $\gamma\subseteq\partial\omega$ of a function $v\in \Hrm^1(\omega)$.
\subsection{Variational and Three-Field formulation}
\newcommand{\jumpS}[1]{\mathop{ \left[\!\!\left[#1\right]\!\!\right] }\nolimits}

Let us consider a connected three-dimensional fracture network $\Omega$ given by the union of open planar fractures $\left\lbrace F_i\right\rbrace _{i \in \mathcal{J}}$, $\mathcal{J}=(0,..., I)$, and surrounded by an impervious rock matrix. This means that the flow, modeled by the Darcy law, only occurs along fractures and through fracture intersections. Given two fractures, their closure intersection is called a \textit{trace}, denoted by $S_m$, $m\in \mathcal{M}=\left\lbrace 1, ..., M\right\rbrace$. The set of all traces in $\Omega$ is $\mathcal{S}$, whereas, for $i\in\mathcal{J}$, the subset $\mathcal{S}_i\subset \mathcal{S}$ contains the traces belonging to the $i$-th fracture; the indexes of traces $S_m\in\mathcal{S}_i$ are collected in the index-set $\mathcal{M}_i$. For each $m\in\mathcal{M}$, the couple $I_{S_m}=\left\lbrace \munderbar{i},\bar{i}\right\rbrace $ denotes the indexes of the two fractures intersecting along $S_m$, with $\munderbar{i}<\bar{i}$. 
The boundary of $\Omega$, denoted by $\partial \Omega$, is split into a Dirichlet part $\Gamma_D$ and a Neumann part $\Gamma_N$, such that $\partial \Omega=\Gamma_D\cup \Gamma_N$, $\Gamma_D \cap \Gamma_N=\emptyset$ and $\Gamma_D \neq \emptyset$. The same holds for fracture boundary $\partial F_i$, having a Dirichlet part $\Gamma_{iD}=\Gamma_D\cap \partial F_i$ and a Neumann part $\Gamma_{iN}=\Gamma_N\cap \partial F_i$. Dirichlet and Neumann boundary conditions on $\partial \Omega$ are expressed by functions $G_D$ and $G_N$, respectively, and their restrictions to $\partial F_i$ are denoted by $G_{iD}$ and $G_{iN}$. 

We are interested in the computation of the hydraulic head $H_i$ on each fracture $F_i\subset \Omega$, which is given by the sum of pressure and elevation. To this end, let us set, on each fracture the following function spaces:
\begin{align*}
&V_i=H_0^1(F_i)=\left\lbrace v \in H^1(F_i):v_{|\Gamma_{iD}}=0\right\rbrace, \quad \forall i \in \mathcal{J},\\
&V_i^D=H_D^1(F_i)=\left\lbrace v \in H^1(F_i):v_{|\Gamma_{iD}}=G_{iD}\right\rbrace , \quad \forall  i \in \mathcal{J},
\end{align*}
and, for each trace $S_m \in \mathcal{S}$, the space $\mathcal{U}^m:=H^{-\frac{1}{2}}(S_m)$ and its dual ${\mathcal{U}^m}'$.
Assuming for the moment that $\Gamma_{iD}\neq \emptyset$ , $\forall i \in \mathcal{J}$, the variational problem describing the distribution of $H$ in $\Omega$ takes the form: for all $i\in\mathcal{J}$, find $H_i=H_i^0+\mathcal{R}_iG_{iD}$ with $\mathcal{R}_iG_{iD}\in V_i^D$ a lifting of the Dirichlet boundary condition and $H_i^0 \in V_i$ such that:
\begin{align}
(\bm{K}_i&\nabla H_i^0,\nabla v_i)_{V_i}=(Q_i,v_i)_{V_i}+\sum_{m \in \mathcal{M}_i}\left\langle \jumpS{\frac{\partial H_i^0}{\partial \hat{\nu}_i^m}},v_{|_{S_m}} \right\rangle_{\mathcal{U}^m,{\mathcal{U}^m}'}+\label{prob_con_prodotti}\\ &\quad +\left\langle G_{iN},v_{|_{\Gamma_{iN}}}\right\rangle_{H^{-\frac{1}{2}}(\Gamma_{iN}),H^{\frac{1}{2}}(\Gamma_{iN})}-(\bm{K}_i\nabla\mathcal{R}_i G_{iD},\nabla v_i)_{V_i} \quad \forall v_i \in V_i,\nonumber
\end{align}
where $\bm{K}_i$ is a uniformly positive definite tensor representing fracture transmissivity, $Q_i$ a  known source term, $\frac{\partial H_i^0}{\partial \hat{\nu}_i^m}=\hat{n}_i^m\cdot \bm{K}_i\nabla H_i^0$ is the hydraulic head co-normal derivative along direction $\hat{n}_i^m$ normal to $S_m \in \mathcal{S}_i$ and $\jumpS{\frac{\partial H_i^0}{\partial \hat{\nu}_i^m}}$ denotes the jump of $\frac{\partial H_i^0}{\partial \hat{\nu}_i^m}$ across $S_m$.

Coupling conditions at the traces for problems on intersecting fractures are the continuity of the hydraulic head and flux conservation, expressed by:
\begin{align}
&{H_{\bar{i}}}_{|_{S_m}}-{H_{\munderbar{i}}}_{|_{S_m}}=0 \qquad \qquad ~~\text{for } \bar{i},\munderbar{i} \in I_{S_m} \quad \forall m \in \mathcal{M} \label{cond_carichi}\\
&\jumpS{\frac{\partial H_{\bar{i}}}{\partial \hat{\nu}_{\bar{i}}^m}}+\jumpS{\frac{\partial H_{\munderbar{i}}}{\partial \hat{\nu}_{\munderbar{i}}^m}}=0  \quad \quad \text{for } \bar{i},\munderbar{i} \in I_{S_m} \quad \forall m \in \mathcal{M}.
\label{cond_flussi}
\end{align}

Let us introduce on each trace $S_m \in \mathcal{S}$ the space $\mathcal{H}^m=H^{\frac{1}{2}}(S_m)$ and its dual ${\mathcal{H}^m}'$, and the quantities $\Psi^m \in \mathcal{H}^m$ and $\Lambda^m \in \mathcal{U}^m$, representing the unknown exact value of the hydraulic head on $S_m$ and of the flux jump across $S_m$, respectively. Coupling condition \eqref{cond_carichi} can be then re-written in a weak form as: $\forall m \in \mathcal{M}$, $\{\munderbar{i},\bar{i}\}=I_{S_m}$
\begin{equation}                             
\begin{aligned}                                        
&\left\langle {H_{\munderbar{i}}}_{|_{S_m}}-\Psi^m,\mu_m\right\rangle_{{\mathcal{H}^m},{\mathcal{H}^m}'} =0 & \forall \mu_m \in {\mathcal{H}^m}',\\
&\left\langle {H_{\bar{i}}}_{|_{S_m}}-\Psi^m,\mu_m\right\rangle_{{\mathcal{H}^m},{\mathcal{H}^m}'} =0 & \forall \mu_m \in {\mathcal{H}^m}',
\end{aligned} \label{condpsi}
\end{equation}
and condition \eqref{cond_flussi} as: $\forall m \in \mathcal{M}$, $\{\munderbar{i},\bar{i}\}=I_{S_m}$
\begin{equation}                             
\begin{aligned}                                 
&\left\langle \jumpS{\frac{\partial H_{\munderbar{i}}}{\partial \hat{\nu}_{\munderbar{i}}^m}}-\Lambda^m,\rho_m\right\rangle_{{\mathcal{U}^m},{\mathcal{U}^m}'} =0 & \forall \rho_m \in {\mathcal{U}^m}',\\
&\left\langle \jumpS{\frac{\partial H_{\bar{i}}}{\partial \hat{\nu}_{\bar{i}}^m}}+\Lambda^m,\rho_m\right\rangle_{{\mathcal{U}^m},{\mathcal{U}^m}'} =0 & \forall \rho_m \in {\mathcal{U}^m}'.
\end{aligned} \label{condlambda}
\end{equation}
Assuming, for the sake of simplicity, that homogenous Dirichlet and Neumann boundary conditions are imposed on $\partial F_i$, $\forall i \in \mathcal{J}$, the Three-Field formulation \cite{Brezzi} of problem \eqref{prob_con_prodotti} takes the form: find $(H_i,\Lambda^m,\Psi^m)\in V_i\times\mathcal{H}^m\times  \mathcal{U}^m$, for all $i\in\mathcal{J}$ and for all $m\in\mathcal{M}$ such that:
\begin{align}
(\bm{K}_i\nabla H_i,\nabla v_i)_{V_i}
&-\sum_{m\in\mathcal{M}_i}\left\langle (-1)^{\chi_{i}^m}\Lambda^m,{v_i}_{|_{S_m}}\right\rangle_{\mathcal{U}^{m},{\mathcal{U}^{m}}'}=(Q_i,v_i)_{V_i}, \qquad \forall v_i\in V_i \label{eqmod1} \\
&\sum_{j\in I_{S_m}}\left\langle {H_j}_{|_{S_m}}-\Psi^m,\mu_m\right\rangle_{{\mathcal{H}^m},{\mathcal{H}^m}'}=0 \qquad\forall \mu_m \in {\mathcal{H}^m}',
\end{align}
with, for $i\in\mathcal{J}$, $m\in\mathcal{M}$, $\chi_{i}^m=1$ if $i=\max(I_{S_m})$ and zero otherwise. 
For a given fracture $F_i$, the second term in equation \eqref{eqmod1} represents the flux entering the fracture through its traces. On each trace $S_m$, $m\in\mathcal{M}$, the flux $\Lambda^m$ is considered positive for fracture $F_{\munderbar{i}}$ and negative for fracture $F_{\bar{i}}$, ensuring the conservation condition. In order to remove the assumption of having a non empty portion of the Dirichlet boundary on each fracture, equation \eqref{eqmod1} can be modified,  as follows: on each fracture $F_i$, $i\in\mathcal{J}$ and on each trace $S_m$, $m\in\mathcal{M}$ find $(H_i,\Lambda^m,\Psi^m)\in V_i\times\mathcal{H}^m\times  \mathcal{U}^m$ such that
\begin{align}
(\bm{K}_i\nabla H_i,\nabla v_i)_{V_i}
&+\alpha\sum_{m\in\mathcal{M}_i}\left(\left({H_i}_{|_{S_m}},{v_i}_{|_{S_m}}\right)_{\mathcal{H}^{m}}
-\left\langle (-1)^{\chi_i^m}\Lambda^m,{v_i}_{|_{S_m}}\right\rangle_{\mathcal{U}^{m},{\mathcal{U}^{m}}'}\right)= \label{eqmod}\\ 
&=\alpha\sum_{m\in\mathcal{M}_i}\left(\Psi^m,{v_i}_{|_{S_m}}\right)_{\mathcal{H}^{m}}+(Q_i,v_i)_{V_i}, \qquad \forall v_i\in V_i \nonumber\\
&\sum_{j\in I_{S_m}}\left\langle {H_j}_{|_{S_m}}-\Psi^m,\mu_m\right\rangle_{{\mathcal{H}^m},{\mathcal{H}^m}'}=0 \qquad\forall \mu_m \in {\mathcal{H}^m}'.\label{split_h}
\end{align}
which, for $\alpha>0$, ensures well posedness of \eqref{eqmod} even if $\Gamma_{iD}=\emptyset$ for all but one fracture.

\subsection{PDE-constrained optimization formulation}
The discretization of the continuity condition \eqref{split_h} would require some sort of mesh conformity at the traces and a discrete inf-sup condition to have well posedness of \eqref{eqmod}-\eqref{split_h}. We want, instead to rewrite problem \eqref{eqmod}-\eqref{split_h} in a new formulation allowing a discretization on arbitrary meshes, from which a viable and robust numerical scheme can be derived, independently of DFN geometrical complexity.
At this aim we transform \eqref{eqmod}-\eqref{split_h} in a PDE-constrained optimization problem, in which a cost functional is introduced in order to enforce the continuity condition on traces. For each fracture $F_i$ and each trace $S_m \in \mathcal{S}_i$ let us introduce the trace operator $\Gamma_i^m:~V_i\rightarrow\mathcal{H}^m$, $\Gamma_i^m(v_i)={v_i}_{|_{S_m}} ~\forall v \in V_i$,
and the cost functional
\begin{equation}
J_i^m(\Lambda^m,\Psi^m)= ||\Gamma_i^m{H_i}(\Lambda^m,\Psi^m)-\Psi^m||^2_{\mathcal{H}^S }, \label{Jim}
\end{equation} 
which expresses the error in the fulfillment of continuity at trace $S_m$.
Let us then introduce, for each fracture $F_i$, $i \in \mathcal{J}$, the spaces
\begin{equation*} 
\mathcal{H}^{\mathcal{M}_i}=\prod_{m \in \mathcal{M}_i}\mathcal{H}^m, \qquad 
\mathcal{U}^{\mathcal{M}_i}=\prod_{m \in \mathcal{M}_i}\mathcal{U}^{m}
\label{spaziMi}
\end{equation*}
and the variables
\begin{equation*}
\Psi_i=\prod_{m\in \mathcal{M}_i} \limits \Psi^m \in \mathcal{H}^{\mathcal{M}_i} , \qquad 
\Lambda_i=\prod_{m \in \mathcal{M}_i}\limits \Lambda^m \in \mathcal{U}^{\mathcal{M}_i}.
\end{equation*}
Setting $\Gamma_i=\prod_{m \in \mathcal{M}_i}\limits \Gamma_i^m$, $\Gamma_i:~V_i\rightarrow\mathcal{H}^{\mathcal{M}_i}$, we define the linear bounded operators 
$A_i:V_i\rightarrow V_i'$, $\mathcal{B}_i:\mathcal{U}^{\mathcal{M}_i}\rightarrow V_i'$, $\mathcal{C}_i:\mathcal{H}^{\mathcal{M}_i}\rightarrow V_i'$ such that
\begin{align}
&\left\langle A_iH_i,v_i\right\rangle _{V_i',V_i}=(\bm{K}_i\nabla H_i,\nabla v_i)_{V_i}+\alpha({\Gamma_iH_i},{\Gamma_iv_i})_{\mathcal{H}^{\mathcal{M}_i}}\quad &v_i \in V_i \label{defA}\\
&\left\langle \mathcal{B}_i\Lambda_i, v_i\right\rangle _{V_i',V_i}=\left\langle (-1)^{\chi_i^m}\Lambda_i,{\Gamma_iv_i}\right\rangle_{\mathcal{U}^{\mathcal{M}_i},{\mathcal{U}^{\mathcal{M}_i}}'} \quad &v_i \in V_i \label{defB}\\
&\left\langle \mathcal{C}_i\Psi_i, v_i\right\rangle _{V_i',V_i}=\alpha(\Psi_i,{\Gamma_iv_i})_{\mathcal{H}^{\mathcal{M}_i}} \quad &v_i \in V_i, \label{defC}
\end{align}
and their adjoints
$A_i^*:V_i\rightarrow V_i'$, $\mathcal{B}_i^*:V_i\rightarrow {\mathcal{U}^{\mathcal{M}_i}}'$ and $ \mathcal{C}_i^*:V_i\rightarrow {\mathcal{H}^{\mathcal{M}_i}}'$.
Defining, then, the spaces 
$$\mathcal{H}=\prod_{m \in \mathcal{M}}\limits\mathcal{H}^m \qquad \mathcal{U}=\prod_{m \in \mathcal{M}}\limits\mathcal{U}^m$$
and the global control variables
$$\Psi=\prod_{m \in \mathcal{M}}\Psi^m\in \mathcal{H} \qquad \Lambda=\prod_{m \in \mathcal{M}}\Lambda^m \in \mathcal{U},$$
a global functional can be introduced as:
\begin{equation}
J(\Lambda,\Psi)=\sum_{i \in \mathcal{J}} J_i(\Lambda_i,\Psi_i)=\sum_{i \in \mathcal{J}}\sum_{m \in \mathcal{M}_i} J_i^m(\Lambda^m,\Psi^m).\label{JJ}
\end{equation}
and problem (\ref{eqmod})-(\ref{split_h}) can be written in the form
\begin{align}
&\min_{(\Lambda,\Psi)} J(\Lambda,\Psi) \text{ subject to } \label{minJm}\\
A_iH_i-&\mathcal{B}_i\Lambda_i-\mathcal{C}_i\Psi_i=Q_i\qquad \forall i \in \mathcal{J} \nonumber.
\end{align}
The following result characterizes the solution to (\ref{minJm}).
\begin{proposition}
The optimal control ($\Lambda,\Psi$) providing the solution to (\ref{minJm}) satisfies, $\forall i \in\mathcal{J}$
\begin{align}
&\Theta_{\mathcal{U}^{\mathcal{M}_i}}^{-1}{\mathcal{B}_i}^*P_i=0 \label{stat_lambda}\\
&\Theta_{\mathcal{H}^{\mathcal{M}_i}}^{-1}{\mathcal{C}_i}^*P_i-\Gamma_iH_i(\Lambda_i,\Psi_i)+\Psi_i=0 \label{stat_psi}
\end{align}
where $P_i\in V_i$ is the solution of 
\begin{equation}
{A_i}^*P_i={\Gamma_i}^*\Theta_{\mathcal{H}^{\mathcal{M}_i}}({\Gamma_i}^*H_i(\Lambda_i,\Psi_i)-\Psi_i) \label{def_p}
\end{equation} and $\Theta_{\mathcal{H}^{\mathcal{M}_i}}:\mathcal{H}^{\mathcal{M}_i}\rightarrow {\mathcal{H}^{\mathcal{M}_i}}'$ and $\Theta_{\mathcal{U}^{\mathcal{M}_i}}:\mathcal{U}^{\mathcal{M}_i}\rightarrow {\mathcal{U}^{\mathcal{M}_i}}'$ are Riesz isomorphisms.
\label{prop1}
\end{proposition}
\begin{proof}
	Let us consider the increments $\delta\Lambda_i$ and $\delta\Psi_i$, concerning the control variables $\Lambda_i$ and $\Psi_i$ respectively, and let us differentiate the cost functional $J(\Lambda,\Psi)$ with respect to the control variables:\\
	\begin{align*}
	\cfrac{\partial J_i}{\partial \Lambda_i}(\Lambda_i+\delta\Lambda_i,\Psi_i)&=2(\Gamma_iH_i(\Lambda_i,\Psi_i)-\Psi_i,\Gamma_iH_i(\delta\Lambda_i,0)_{\mathcal{H}^{\mathcal{M}_i}})=\nonumber\\
	&=2\left\langle {A_i}^*P_i,A_i^{-1}\mathcal{B}_i\delta\Lambda_i\right\rangle_{V_i',V_i}
	=2(\Theta_{\mathcal{U}^{\mathcal{M}_i}}^{-1}{\mathcal{B}_i}^*P_i,\delta\Lambda_i)_{\mathcal{U}^{\mathcal{M}_i}}\nonumber
	\end{align*}
	\begin{align*}
	\cfrac{\partial J_i}{\partial\Psi_i}(\Lambda_i,\Psi_i+&\delta\Psi_i)=2\left(\Gamma_iH_i(\Lambda_i,\Psi_i)-\Psi_i,\Gamma_iH_i(0,\delta\Psi_i)-\delta\Psi_i\right) _{\mathcal{H}^{\mathcal{M}_i}}=\nonumber\\
	&=2\left\langle {A_i}^*P_i, A_i^{-1}\mathcal{C}_i\delta\Psi_i\right\rangle_{V_i',V_i}
-2\left(\Gamma_iH_i(\Lambda_i,\Psi_i)-\Psi_i,\delta\Psi_i\right) _{\mathcal{H}^{\mathcal{M}_i}}=\nonumber\\[0.2em]
	&=2\left(\Theta_{\mathcal{H}^{\mathcal{M}_i}}^{-1}{\mathcal{C}_i}^*P_i-\Gamma_iH_i(\Lambda_i,\Psi_i)+\Psi_i,\delta\Psi_i\right) _{\mathcal{H}^{\mathcal{M}_i}},
	\end{align*}
	and this yields the thesis.
	\end{proof}

The derivatives computed in the proof of Proposition \ref{prop1} represent the Frechet derivative of the Lagrangian function associated to problem (\ref{minJm}), for which the variable $P_i \in V_i$ is the Lagrangian multiplier on fracture $F_i$. The solution to problem (\ref{minJm}) can then be found by imposing stationarity conditions for the Lagrangian. Nevertheless, as we will show later, when dealing with huge and complex DFNs it might be computationally more convenient to minimize $J(\Lambda,\Psi)$ using an iterative method, such as the conjugate gradient method.
Starting from the derivatives computed in Proposition \ref{prop1} let us consider the following quantities, for each $i\in\mathcal{J}$:
\begin{align}
&\delta \Lambda_i=\Theta_{\mathcal{U}^{\mathcal{M}_i}}^{-1}{\mathcal{B}_i}^*P_i, &\delta \Lambda=\sum_{i \in \mathcal{J}}\delta \Lambda_i,\\
&\delta \Psi_i=\Theta_{\mathcal{H}^{\mathcal{M}_i}}^{-1}{\mathcal{C}_i}^*P_i-\Gamma_iH_i(\Lambda_i,\Psi_i)+\Psi_i, &\delta \Psi=\sum_{i \in \mathcal{J}}\delta \Psi_i.
\end{align}
Let then $\delta H_i=H_i(\delta \Lambda_i, \delta \Psi_i)$ and $\delta P_i$ be the solutions of
\begin{align}
&A_i\delta H_i=\mathcal{B}_i\delta \Lambda_i+ \mathcal{C}_i \delta \Psi_i, \quad \forall i\in \mathcal{J}\\
&{A_i}^*\delta P_i={\Gamma_i}^*\Theta_{\mathcal{H}^{\mathcal{M}_i}}(\Gamma_i\delta H_i-\delta \Psi_i) , \quad \forall i\in \mathcal{J}.
\end{align}

\begin{proposition}
	Given the control variable $W:=(\Lambda,\Psi)$, let us increment it by a step $\zeta\delta W$, with $\delta W:=(\delta \Lambda,\delta \Psi)$. The steepest descent method corresponds to the stepsize
	\begin{equation}
	\zeta=\cfrac{\sum_{i \in \mathcal{J}}\limits\left[ (\delta \Lambda_i, \delta \Lambda_i)_{\mathcal{U}^{\mathcal{M}_i}}+(\delta \Psi_i,\delta \Psi_i)_{\mathcal{H}^{\mathcal{M}_i}}\right]}{\sum_{i \in \mathcal{J}}\limits\left[\left\langle \mathcal{B}_i\delta \Lambda_i+\mathcal{C}_i\delta \Psi_i,\delta P_i\right\rangle _{V_i',V_i}-(\Gamma_i\delta H_i,\delta \Psi_i)_{\mathcal{H}^{\mathcal{M}_i}}+(\delta \Psi_i, \delta \Psi_i)_{\mathcal{H}^{\mathcal{M}_i}}\right] }
	\end{equation} 
	\label{prop2}
\end{proposition}
\begin{proof}
	It is sufficient to set to zero the derivative $\cfrac{\partial J(W+\zeta\delta W)}{\partial \zeta}$.
\begin{align*}
J(W+\zeta\delta W)&=
J(W)+2\zeta \sum_{i \in \mathcal{J}}(\Gamma_iH_i(\Lambda_i,\Psi_i)-\Psi_i,\Gamma_iH_i(\delta \Lambda_i,\delta\Psi_i)-\delta\Psi_i)_{\mathcal{H}^{\mathcal{M}_i}}+\\[-0.4em]
&\quad +\zeta^2\sum_{i \in \mathcal{J}}||\Gamma_iH_i(\delta \Lambda_i,\delta\Psi_i)-\delta\Psi_i||_{\mathcal{H}^{\mathcal{M}_i}}^2
\end{align*}
\begin{align*}
\cfrac{\partial J(W+\zeta\delta W)}{\partial \zeta}&=2\sum_{i \in \mathcal{J}}(\Gamma_iH_i(\Lambda_i,\Psi_i)-\Psi_i,\Gamma_iH_i(\delta \Lambda_i,\delta \Psi_i)-\delta\Psi_i)_{\mathcal{H}^{\mathcal{M}_i}}+\nonumber\\[-0.2em]
&\quad+2\zeta\sum_{i \in \mathcal{J}}||\Gamma_iH_i(\delta \Lambda_i,\delta\Psi_i)-\delta\Psi_i||_{\mathcal{H}^{\mathcal{M}_i}}^2=0
\end{align*}
	\begin{equation*}
	\zeta=\cfrac{\sum_{i \in \mathcal{J}}\limits(\Gamma_iH_i(\Lambda_i,\Psi_i)-\Psi_i,\Gamma_iH_i(\delta \Lambda_i,\delta \Psi_i)-\delta\Psi_i)_{\mathcal{H}^{\mathcal{M}_i}}}{{\sum_{i \in \mathcal{J}}\limits||\Gamma_iH_i(\delta \Lambda_i,\delta\Psi_i)-\delta\Psi_i||^2_{\mathcal{H}^{\mathcal{M}_i}}}}
\end{equation*} 
from which the thesis follows.
\end{proof}
\section{Discretization}
\label{DiscreteModel}

In this section we introduce suitable space dicretizations on fractures and traces, and we derive the corresponding discrete formulation of the problem. In the following, we will denote by lower-case letters the finite dimensional approximation of the continuous variables with respect to suitable bases. The same notation will be used for the discrete functions and for the corresponding vectors of degrees of freedom (DOFs), the meaning being clear from the context.

Let us build a triangular mesh on each fracture $F_i$, $i \in \mathcal{J}$, non conforming to the traces on the fracture, and let us define, on this mesh, suitable finite elements basis functions for the hydraulic head $\left\lbrace \varphi_{i,k}\right\rbrace _{k=1,...,N_H^i}$, with $N_H^i$ denoting the number of DOFs on the $i$-th fracture. The approximation of $H_i$ with respect to this basis is
\begin{equation}
h_i=\sum_{k=1}^{N_H^i} h_{i,k}\varphi_{i,k},
\end{equation}
where $h_{i,k}$ are the values of the degrees of freedom. For each trace $S_m \in \mathcal{S}$ let us build two different meshes and let us consider two bases $\left\lbrace \eta_{k}^m\right\rbrace _{k=1,...,N_{\Lambda}^m}$ and $\left\lbrace \theta_{k}^m\right\rbrace _{k=1,...,N_{\Psi}^m}$, with $N_{\Lambda}^m$ and $N_{\Psi}^m$ denoting the number of DOFs on the $m$-th trace, respectively for $\Lambda^m$ and $\Psi^m$. It is worth highlighting that neither a unique discretization nor the same basis is required for the two control variables. The discrete control variables are
\begin{equation}
\lambda^m=\sum_{k=1}^{N_{\Lambda}^m} \lambda_{k}^m\eta_{k}^m,
\qquad
\psi^m=\sum_{k=1}^{N_{\Psi}^m} \psi_{k}^m\theta_{k}^m,
\end{equation}
with $\lambda_k^m$ and $\psi_k^m$ denoting the values assigned to the DOFs.\\

Let us then define, for each fracture $F_i$ the vector of the hydraulic head DOFs $h_i \in \mathbb{R}^{N_H^i}$ obtained collecting column-wise the relative DOFs, and matrix $\bm{A}_i$ defined as
\begin{equation}
\bm{A_i} \in \mathbb{R}^{N_H^i \times N_H^i}, \quad 
(A_i)_{kl}=\int_{F_i}\bm{K}_i\nabla \varphi_{i,k}\nabla \varphi_{i,l}~dF_i+\alpha\int_{\mathcal{S}_i}{\varphi_{i,k}}_{|_{{\mathcal{S}_i}}}{\varphi_{i,l}}_{|_{\mathcal{S}_i}}~dS.
\end{equation} 
For each trace $S_m \in \mathcal{S}$ let us consider the vectors of control variable DOFs $\lambda^m \in \mathbb{R}^{N_{\Lambda}^m}$ and $\psi^m \in \mathbb{R}^{N_{\Psi}^m}$, obtained once again collecting column-wise the corresponding DOFs. Furthermore let us introduce the following matrices, defined on each trace $S_m$ of each fracture $F_i$, $\forall i \in \mathcal{J}$, $\forall m \in \mathcal{M}_i$:
\begin{align}
&\bm{\mathcal{B}_i^m} \in \mathbb{R}^{N_H^i \times N_{\Lambda}^m}, \quad (\mathcal{B}_i^m)_{kl}=(-1)^{\chi_i^m}\int_{S_m} {\varphi_{i,k}}_{|_{S_m}}~\eta_{l}^m~dS \label{Bim}\\
&\bm{\mathcal{C}_i^m} \in \mathbb{R}^{N_H^i \times N_{\Psi}^m},\quad (\mathcal{C}_i^m)_{kl}=\alpha \int_{S_m} {\varphi_{i,k}}_{|_{S_m}}~\theta_{l}^m~dS \label{Cim}
\end{align}
and the matrices $\bm{\mathcal{B}_i}$ and $\bm{\mathcal{C}_i}$ on $F_i$, obtained collecting respectively the matrices $\bm{\mathcal{B}_i^m}$ and $\bm{\mathcal{B}_i^m}$ for increasing values of indices $m \in \mathcal{M}_i=(m_1,... m_{M{_i}})$
\begin{equation}
\bm{\mathcal{B}_i}=\left[ \bm{\mathcal{B}_i^{m_1}}, \bm{\mathcal{B}_i^{m_2}},\cdots, \bm{\mathcal{B}_i^{m_{M_i}}} \right], \qquad
\bm{\mathcal{C}_i}=\left[ \bm{\mathcal{C}_i^{m_1}}, \bm{\mathcal{C}_i^{m_2}},\cdots, \bm{\mathcal{C}_i^{m_{M_i}}} \right].
\end{equation}
Finally let us define the vectors
\begin{equation}
\lambda_i=
\begin{bmatrix} \lambda_i^{m_1}\\\lambda^{m_2}\\ \vdots\\\lambda^{m_{M_i}}
\end{bmatrix} \in \mathbb{R}^{N_\Lambda^{\mathcal{M}_i}}, \qquad
\psi_i=
\begin{bmatrix}
\psi_i^{m_1}\\\psi^{m_2}\\\vdots\\\psi^{m_{M_i}} \end{bmatrix} \in \mathbb{R}^{N_{\Psi}^{\mathcal{M}_i}}, \label{lambdai_psii}
\end{equation}
with $ N_{\Lambda}^{\mathcal{M}_i}=\sum_{m\in \mathcal{M}_i}\limits N_{\Lambda}^m$ and $N_{\Psi}^{\mathcal{M}_i}=\sum_{m\in \mathcal{M}_i}\limits N_{\Psi}^m$.
We are then able to write the discrete matrix formulation of the constraints equation in problem (\ref{minJm})
\begin{equation}
\bm{A_i}h_i-\bm{\mathcal{B}_i}\lambda_i-\bm{\mathcal{C}_i}\psi_i=q_i.
\label{disc_i}
\end{equation}
where $q_i \in \mathbb{R}^{N_H^i}$ corresponds to the discrete source term on $F_i$.\\

In view of a global formulation over the whole DFN, a global vector containing the head's DOFs is built as
\begin{equation}
h=
\begin{bmatrix}
h_1 \\ h_2 \\ \vdots\\h_I
\end{bmatrix}\in \mathbb{R}^{N_H^F},\label{h}
\end{equation}
where $N_H=\sum_{i\in \mathcal{J}}\limits N_H^i$. Global vectors for the control variables are obtained concatenating column-wise vectors $\left\lbrace \lambda^m\right\rbrace _{m \in \mathcal{M}}$ and $\left\lbrace \psi^m\right\rbrace _{m \in \mathcal{M}}$, namely
\begin{equation}
\lambda=
\begin{bmatrix}
\lambda^1\\\lambda^2\\\vdots\\\lambda^M
\end{bmatrix} \in \mathbb{R}^{N_{\Lambda}}, \qquad  \psi=\begin{bmatrix}
\psi^1\\\psi^2\\\vdots\\\psi^M
\end{bmatrix} \in \mathbb{R}^{N_{\Psi}},
\end{equation}
where $N_{\Lambda}=\sum_{m\in \mathcal{M}}\limits N_{\Lambda}^m$ and $N_{\Psi}=\sum_{m\in \mathcal{M}}\limits N_{\Psi}^m$.
Let us define, $\forall i \in \mathcal{J}$, matrices
\begin{align}
&\bm{\mathcal{B}_i^m}=\bm{0} \in \mathbb{R}^{N_H^i \times N_{\Lambda}^m} ~\forall m \notin \mathcal{M}_i\\
&\bm{\mathcal{C}_i^m}=\bm{0} \in \mathbb{R}^{N_H^i \times N_{\Psi}^m} ~\forall m \notin \mathcal{M}_i.
\end{align} 
and, recalling definitions in \eqref{Bim} and \eqref{Cim}, we build:
\begin{align}
&\bm{\mathcal{B}_i^{\mathcal{M}}}=\left[ \bm{\mathcal{B}_i^1}, \bm{\mathcal{B}_i^2}, \cdots, \bm{\mathcal{B}_i^M} \right]  \in \mathbb{R}^{N_H^i \times N_{\Lambda}},\label{Bim_new}\\
&\bm{\mathcal{C}_i^{\mathcal{M}}}=\left[ \bm{\mathcal{C}_i^1}, \bm{\mathcal{B}_i^2}, \cdots, \bm{\mathcal{C}_i^M} \right]  \in \mathbb{R}^{N_H^i \times N_{\Lambda}}.\label{Cim_new}
\end{align}
and
\begin{equation}
\bm{\mathcal{B}}=
\begin{bmatrix}
\bm{\mathcal{B}_1^{\mathcal{M}}}\\\bm{\mathcal{B}_2^{\mathcal{M}}}\\\vdots\\\bm{\mathcal{B}_I^{\mathcal{M}}}
\end{bmatrix}\in \mathbb{R}^{N_H\times N_{\Lambda}}, \qquad 
\bm{\mathcal{C}}=
\begin{bmatrix}
\bm{\mathcal{C}_1^{\mathcal{M}}}\\\bm{\mathcal{C}_2^{\mathcal{M}}}\\\vdots\\\bm{\mathcal{C}_I^{\mathcal{M}}}
\end{bmatrix}\in \mathbb{R}^{N_H\times N_{\Psi}}\label{BC},
\end{equation}
such that the global discrete form of the constraints equation becomes
\begin{equation}
\bm{A}h-\bm{\mathcal{B}}\lambda-\bm{\mathcal{C}}\psi=q \label{eqglob}
\end{equation}
where $
\bm{A}=\text{diag}(\bm{A_1}~\bm{A_2}~...~\bm{A_I}) \in \mathbb{R}^{N_H \times N_H}$
and $q=(q_1^T~q_2^T~...~q_I^T)^T \in \mathbb{R}^{N_H}.$\\
\indent The discrete functional is obtained from equation (\ref{Jim}) by use of the discrete functions and of $L^2$ norms in place of $\mathcal{H}^m$ norms, this yielding, for $i \in \mathcal{J}$, $m \in \mathcal{M}_i$, to 
\begin{equation}
\tilde{J}_i^m(\lambda^m,\psi^m)=||{h_i}_{|_{S_m}}(\lambda^m,\psi^m)-\psi^m||^2_{L^2}.
\end{equation}
Defining the matrices
\begin{align}
&\bm{G_{i}^{h,m}} \in \mathbb{R}^{N_H^i\times N_H^i}, \quad \left(G_{i}^{h,m}\right)_{kl}=\int_{S_m}{\varphi_{i,k}}_{|_{S_m}}
{\varphi_{i,l}}_{|_{S_m}}~dS  \label{Gh_i_m} \\
&\bm{G^{\psi,m}} \in \mathbb{R}^{N_{\Psi}^m \times N_{\Psi}^m}, \quad \left(G^{\psi,m}\right)_{kl}=\int_{S_m}{\theta_{k}^m}{\theta_{l}^m}~dS 
\label{Gpsi_i_m}
\end{align}
and
\begin{align}
&\bm{G_i^h}=\sum_{m\in \mathcal{M}_i}\bm{G_{i}^{h,m}} \in \mathbb{R}^{N_H^i\times N_H^i},\\
&\bm{G_i^\psi}=\text{diag}(\bm{G^{\psi,m_1}}~\bm{G^{\psi,m_2}}~...~\bm{G^{\psi,m_{M_i}}}) \in \mathbb{R}^{N_{\Psi}^{\mathcal{M}_i}\times {N_{\Psi}^{\mathcal{M}_i}}}, \label{Gpsi_i}
\end{align}
the discrete cost functional relative to the $i$-th fracture takes the form:
\begin{equation}
\tilde{J}_i(\lambda_i,\psi_i)=
h_i^T\bm{G_i^{h}}h_i+\psi_i^T\bm{G_i^{\psi}}\psi_i-h_i^T\bm{\mathcal{C}_i}\psi_i-\psi_i^T\bm{\mathcal{C}_i}^Th_i \quad \forall i \in \mathcal{J},
\label{Jtilde}
\end{equation}
where $h_i=h_i(\lambda_i, \psi_i)$.
Finally, introducing the matrices
\begin{align}
&\bm{G^h}=\text{diag}(\bm{G_1^h}~\bm{G_2^h}~...~\bm{G_I^h})\in \mathbb{R}^{N_H\times N_H}\\
&\bm{G^{\psi}}=2 \left( \text{diag}(\bm{G_1^\psi}~\bm{G_2^\psi}~...~\bm{G_M^\psi})\right)\in \mathbb{R}^{N_{\Psi}\times N_{\Psi}}. \label{Gpsi}
\end{align}
The global discrete matrix formulation of the cost functional is obtained as
\begin{equation}
\tilde{J}(\lambda,\psi)=
h^T\bm{G^h}h+\psi^T\bm{G^{\psi}}\psi-h^T\bm{\mathcal{C}}\psi-\psi^T\bm{\mathcal{C}}^Th \label{funzionale},
\end{equation}
with $h=h(\lambda, \psi)$,
thanks to which we obtain the following global discrete matrix formulation of the problem describing the subsurface flow through a DFN:
	\begin{equation}
	\min_{(\lambda,\psi)}\tilde{J}(\lambda,\psi) \text{ subject to } (\ref{eqglob}). \label{minJtilde}
	\end{equation}
Exploiting the linearity of the constraints we derive the following unconstrained minimization problem equivalent to (\ref{minJtilde}), replacing $h=h(\lambda,\psi)=\bm{A}^{-1}\bm{\mathcal{B}}\lambda+\bm{A}^{-1}\bm{\mathcal{C}}\psi+\bm{A}^{-1}q$ in the definition of the functional:
\begin{equation}
\min_{(\lambda,\psi)}\tilde{J}^*(\lambda, \psi)\label{Jstar}
\end{equation}
where
\begin{equation}\tilde{J}^*(\lambda,\psi)=\begin{bmatrix}
\lambda^T & \psi^T
\end{bmatrix}
\hat{\bm{G}}                             
\begin{bmatrix}                             
\lambda \\ \psi\end{bmatrix} + 2d^T                            
\begin{bmatrix}
\lambda \\ \psi
\end{bmatrix}+ q^T \begin{bmatrix}
\bm{A}^{-T}\bm{G^h}\bm{A}^{-1}\end{bmatrix}q \label{sistemone},
\end{equation}
\begin{equation}\hat{\bm{G}}=
\begin{bmatrix}
\bm{\mathcal{B}}^T\bm{A}^{-T}\bm{G^h}\bm{A}^{-1}\bm{\mathcal{B}}\quad & \bm{\mathcal{B}}^T\bm{A}^{-T}(\bm{G^h}\bm{A}^{-1}-\bm{I})\bm{\mathcal{C}}\\\\
\bm{\mathcal{C}}^T(\bm{A}^{-T}\bm{G^h}-\bm{I}) \bm{A}^{-1}\bm{\mathcal{B}}\quad & \bm{G^{\psi}}+\bm{\mathcal{C}}^T(\bm{A}^{-T}\bm{G^h}\bm{A}^{-1}-\bm{A}^{-T}-\bm{A}^{-1})\bm{\mathcal{C}}
\end{bmatrix} ,  \label{matrice}
\end{equation}
and
\begin{equation}
d^T= q^T[ \bm{A}^{-T}\bm{G^h}\bm{A}^{-1}\bm{\mathcal{B}} \quad \bm{A}^{-T}(\bm{G^h}\bm{A}^{-1}-\bm{I})\bm{\mathcal{C}}]. \label{d}
\end{equation}

\section{Existence and uniqueness of the discrete solution}
\label{DiscrWellPosedness}
\newcommand{\Mkkt}{\bm{M}_{\text{KKT}}}
The system of optimality conditions (KKT-conditions) for problem~\eqref{minJtilde} can be written as:
\begin{equation}
\begin{bmatrix}
\bm{G^h} & 0 & -\bm{\mathcal{C}} & \bm{A}^T\\
0 & 0 & 0 & -\bm{\mathcal{B}}^T\\
-\bm{\mathcal{C}}^T & 0 & \bm{G^{\psi}} & -\bm{\mathcal{C}}^T\\
\bm{A} & -\bm{\mathcal{B}} & -\bm{\mathcal{C}} & 0\\
\end{bmatrix}
\begin{bmatrix}
h\\ \lambda \\ \psi \\ -p
\end{bmatrix}=
\begin{bmatrix}
0\\0\\0\\q
\end{bmatrix},
\label{KKT}
\end{equation}
where $p$ is the array of Lagrange multipliers. Grouping matrices and vectors as follows:
\begin{equation}
\label{matrDef}
\bm{\mathcal{G}}=\begin{bmatrix}
\bm{G^h} & 0 & -\bm{\mathcal{C}}\\
0 & 0 & 0 \\
-\bm{\mathcal{C}}^T & 0 & \bm{G^{\psi}}
\end{bmatrix},
\quad
\bm{\mathcal{A}}=\begin{bmatrix}
\bm{A} & -\bm{\mathcal{B}} & -\bm{\mathcal{C}}
\end{bmatrix},
\quad
w=\begin{bmatrix}
h\\
\lambda\\
\psi\\
-p
\end{bmatrix},
\quad
q_\text{KKT}=\begin{bmatrix}
0\\
0\\
0\\
q
\end{bmatrix},
\end{equation}
the KKT system, can be compactly rewritten as:
\begin{equation}
\Mkkt=\begin{bmatrix}
\bm{\mathcal{G}} & \bm{\mathcal{A}}^T\\
\bm{\mathcal{A}} & \bm{O}
\end{bmatrix},
\qquad
\Mkkt w=q_\text{KKT}.
\label{Mkkt}
\end{equation}
\begin{proposition}
\label{discretewellposedness}
Matrix $\Mkkt$ in \eqref{Mkkt} is non singular and the unique solution of problem \eqref{Mkkt} is equivalent to the solution of \eqref{minJtilde}.
\end{proposition}
The proof of Proposition \ref{discretewellposedness} is based on the following lemma:
\begin{lemma}
\label{lemma1}
Matrix $\bm{\mathcal{A}}$ in \eqref{matrDef} is full row-rank and, being $\bm{Z}$ a matrix formed collecting column-wise vectors $z_k$, $k=1,\ldots,N_{\Lambda}+N_{\Psi}$, forming a basis of $\ker{(\bm{\mathcal{A}})}$, matrix $\bm{Z}^T\bm{\mathcal{G}}\bm{Z}$, for $\bm{\mathcal{G}}$ as in \eqref{matrDef}, is symmetric positive definite.
\end{lemma}
\begin{proof}
Matrix $\bm{\mathcal{A}}$ is full row rank by construction, as matrix $\bm{A}$ in \eqref{matrDef} is non-singular. The size of $\ker{(\bm{\mathcal{A}})}$ is thus $N_{\Lambda}+N_{\Psi}$. Let us choose the canonical basis for $\mathbb{R}^{N_{\Lambda}+N_{\Psi}}$ and let us take the $k-th$ element of such basis, denoted by $e_k$, $k=1,\ldots,N_\Lambda+N_\Psi$. The corresponding element $z_k\in\ker{(\bm{\mathcal{A}})}$ has the following structure:
\begin{displaymath}
z_k=\begin{bmatrix}
\bm{A}^{-1}\begin{bmatrix}
\bm{\mathcal{B}} & \bm{\mathcal{C}}
\end{bmatrix}
e_k\\
e_k
\end{bmatrix}.
\end{displaymath}
Let us now choose $1\leq k\leq N_{\Lambda}$, thus giving
$$z_k=\begin{bmatrix} \bm{A}^{-1}\bm{\mathcal{B}}e_k\\ e_k \end{bmatrix}:=\begin{bmatrix}\bar{h}_k\\ e_k\end{bmatrix},$$
with $\bar{h}_k$ being different from zero on at least one trace of the network, in virtue of equation \eqref{disc_i}, given the non singularity of $\bm{A}$ and being $\mathcal{B} e_k\neq 0$. Thus it can be easily concluded that $z_k^T \bm{\mathcal{G}}z_k= \bar{h}_k^T \bm{G^h} \bar{h}_k> 0$, for all $1\leq k\leq N_{\Lambda}$.

If now $N_{\Lambda}+1\leq k\leq N_{\Psi}$, it is 
$$z_k=\begin{bmatrix}\bm{A}^{-1}\bm{\mathcal{C}}e_k\\ e_k \end{bmatrix}:=\begin{bmatrix}\bar{\bar{h}}_k\\ e_k\end{bmatrix},$$
and, correspondingly to $e_k$, there is a unique index $m^*\in\mathcal{M}$ such that $\psi^{m^*}\neq 0$, being, instead $\lambda\equiv 0$. Let us select the two fractures, $F_i$ and $F_j$ such that $\{i,j\}=I_{S_{m^\star}}$
If the networks contains more than two fractures, at least one of these fractures, say $F_i$, has more than one trace and on $F_i$ the discrete constraint equation reads: $\forall j=1,\ldots,N^i_H$
\begin{displaymath}
\int_{F_i} \bm{K}_i \nabla \bar{\bar{h}}_k \nabla \varphi_j \ dF_i + \alpha \!\!\! \sum_{\substack{m\in\mathcal{M}_i,\\ m\neq m^*}}\int_{S_m}\bar{\bar{h}}_{k|S_{m}} \varphi_{j|S_m} \ dS=\alpha \int_{S_{m^*}}\!\! \left(\bar{\bar{h}}_{k|S_{m^*}}-\psi^{m^*}\right) \varphi_{j|S_m} \ dS.
\end{displaymath}
If now we assume $\bar{\bar{h}}_{k|S_{m^*}}=\psi^m\neq 0$ we obtain through the constraint equation $\bar{\bar{h}}_k=0$, which is an absurd. If there are only two fractures in the network, a similar conclusion can be derived, since at least one of the two fractures has a non empty portion of the Dirichlet boundary. Then we have $z_k^T \bm{\mathcal{G}}z_k\geq \|\bar{\bar{h}}_{k|S_{m^*}}-\psi^m\|^2>0$ for all $N_{\Lambda}+1\leq k\leq N_{\Psi}$.

Thus, for any $k=1,\ldots, N_{\Lambda}+N_{\Psi}$, $z_k\not\in\ker{(\bm{\mathcal{G}})}$ and the vector space $\mathcal{Z}=\text{span}\{z_1,\ldots,z_{N_{\Lambda}+N_{\Psi}}\}$ is a subspace of $\text{Im}(\bm{\mathcal{G}})$. For each $y\in \mathcal{Z}$ we have $y=\bm{Z}v$, for $v\in\mathbb{R}^{N_{\Lambda}+N_{\Psi}}$ and we can therefore conclude that $y^T\bm{\mathcal{G}}y>0$, or equivalently $v^T\bm{Z}^T\bm{\mathcal{G}}\bm{Z}v>0$.
\end{proof}
The proof of Proposition \ref{discretewellposedness} follows from Lemma \ref{lemma1} and classical arguments of quadratic programming.

\section{Problem resolution}
\label{prob_solve}
\newcommand{\Ghat}{\hat{\bm{G}}}
Solving the KKT-system \eqref{KKT} in order to compute an approximation of the hydraulic head in $\Omega$ might not be a viable option for large networks, for which matrix $\bm{M}_{\text{KKT}}$ would be extremely large and, likely, ill-conditioned. It is convenient, instead, to solve the unconstrained minimization problem \eqref{Jstar} via a gradient method, which also results in an algorithm well suited for parallel implementation on parallel computing machines. 
Let us rewrite the cost functional (\ref{Jstar}) in a compact form as 
\begin{equation}
\tilde{J}^*=w^T\hat{\bm{G}}w+2d^Tw+q^T \begin{bmatrix}
\bm{A}^{-T}\bm{G^h}\bm{A}^{-1}\end{bmatrix}q, \label{Jcompact}
\end{equation}
where $w=[\lambda^T,\psi^T]^T$, and let us observe that $\nabla \tilde{J}^*=\Ghat w+d$.\\
\begin{algorithm}
	\caption{Preconditioned conjugate gradient method applied to $\hat{\bm{G}}w+d=0$}
	\label{gradcon_prec}
	\begin{algorithmic}[1]
		\STATE Guess $w_0=[\lambda_0^T, \psi_0^T]^T$\\
		\STATE $r_0=\hat{\bm{G}}w_0+d$
		\STATE solve $\bm{P}z_0=r_0$\\
		\STATE  set $\delta w_0=-z_{0}$ and $k=0$;
		\WHILE{$\|r_k\|>0$}
		\STATE $\zeta_k=\cfrac{r_k^Tz_k}{\delta w_{k}^T\hat{\bm{G}}\delta w_{k}}$;\label{step1}\\
		\STATE $w_{k+1}=w_k+\zeta_k\delta w_{k}$;\\
		\STATE $r_{k+1}=r_k+\zeta_k\hat{\bm{G}}\delta w_{k}$; \label{step2}\\ 
		\STATE solve $\bm{P}z_{k+1}=r_{k+1}$;\label{stepp}\\
		\STATE $\beta_{k+1}=\cfrac{r_{k+1}^Tz_{k+1}}{r_{k}^Tz_{k}}$;\\
		\STATE $\delta w_{k+1}=-z_{k+1}+\beta_{k+1}\delta w_k$;\\
		\STATE $k=k+1$;
		\ENDWHILE
	\end{algorithmic}
\end{algorithm}
Algorithm~\ref{gradcon_prec} reports the steps of the application of the preconditioned conjugate gradient scheme to the resolution of $\Ghat w+d=0$, with a preconditioner $\bm{P}$. It is to remark that, for any vector $w=\left[\lambda^T,\psi^T \right]^T,$ $\lambda\in\mathbb{R}^{N_\Lambda}$, $\psi\in\mathbb{R}^{N_\Psi}$, the computation of $\Ghat w$, as at steps \ref{step1}, \ref{step2} of Algorithm~\ref{gradcon_prec}, does not require the inversion of matrix $\bm{A}$. In particular it only involves the resolution of linear systems defined independently on each fracture in $\Omega$, which, therefore, can be performed in parallel. Indeed, setting 
\begin{displaymath}
h=\bm{A}^{-1}(\bm{\mathcal{B}}\lambda+\bm{\mathcal{C}} \psi), \qquad
p=\bm{A}^{-T}(\bm{G^h} h-\bm{\mathcal{C}}\psi),
\end{displaymath}
which can be computed locally on the fractures thanks to the structure of the involved matrices, we have:
\begin{displaymath}
\Ghat w= 
\begin{bmatrix}
\bm{\mathcal{B}}^T h\\
\bm{G^\psi}\psi+\bm{\mathcal{C}^T} p-\bm{\mathcal{C}}^T h
\end{bmatrix}.
\end{displaymath}
The choice of preconditioner $\bm{P}$ is of great importance for the performances of the method. Given the structure of matrix $\Ghat$ in \eqref{matrice}, neglecting off-diagonal terms and simplifying the structure of the bottom-right term, a possible choice is the following:
\begin{equation}
\bm{P_f}=\begin{bmatrix}
\bm{\mathcal{B}}^T\bm{A}^{-T}\bm{G^h}\bm{A}^{-1}\bm{\mathcal{B}} & \bm{\mathcal{O}}\\
\bm{\mathcal{O}} & \bm{G^{\psi}}
\end{bmatrix}\label{Pfull}
\end{equation}
which provides very good results, as shown in the following section. Unfortunately the efficient, parallel, application of such preconditioner, such as at step~\ref{stepp} of Algorithm~\ref{gradcon_prec}, would require inner loops of a gradient based scheme, analogously to what done to solve the main problem. For this reason a new preconditioner is introduced, further simplifying the structure of $\Ghat$, and preconditioner $\bm{P_f}$ is retained only as a term of comparison. The new preconditioner is defined only extracting $M$ block-diagonal terms of size $N^m_\Lambda$, $m=1,\ldots,M$ from matrix $\bm{\mathcal{D}}:=\bm{\mathcal{B}}^T\bm{A}^{-T}\bm{G^h}\bm{A}^{-1}\bm{\mathcal{B}}$: denoting by $N^{[m]}_\Lambda=\sum_{\ell=1}^{m} N^\ell_\Lambda$, matrix $\bm{\mathcal{D}_m}$ is obtained taking the elements at rows and columns $N^{[m-1]}_\Lambda, \ldots, {N}^{[m]}_\Lambda$ of $\bm{\mathcal{D}}$, and:
\begin{equation}
\bm{P_d}=\begin{bmatrix}
\text{diag}(\bm{\mathcal{D}_1},\ldots,\bm{\mathcal{D}_M})  & \bm{\mathcal{O}}\\
\bm{\mathcal{O}} & \bm{G^{\psi}}
\end{bmatrix}.\label{Pdiag}
\end{equation}

\section{Numerical results}
\label{NumRes}
\newcommand{\dlam}{\bm{\delta_\lambda}}
\newcommand{\dpsi}{\bm{\delta_\psi}}
\newcommand{\dhi}{\bm{\delta_{h}}}
\newcommand{\EhL}{\mathcal{E}^h_{L^2}}
\newcommand{\EhH}{\mathcal{E}^h_{H^1}}
\newcommand{\Elam}{\mathcal{E}^{\lambda}_{L^2}}
Here some numerical results are reported to describe the behavior of the proposed numerical method. Three different networks of increasing complexity are considered: first the hydraulic head is computed on a small network of three fractures, comparing the obtained solution to the available known exact solution; then a slightly bigger network of ten fractures is analyzed in order to highlight and discuss the characteristics of the method in a more realistic, yet synthetic, framework, and finally some results are presented on a complex network counting slightly less than four hundred fractures, obtained starting from realistic input data. More details on the networks used in the simulations are reported in Table~\ref{DFNinfo}.

First order Lagrangian finite elements are used to describe the hydraulic head on the fractures, on meshes of triangular elements non conforming to the traces and independently built on each fracture. Additional enrichment functions are used on the elements intersected by the traces, according to the eXtended Finite Element framework (see \cite{BPSb}), in order to reproduce jumps of the co-normal derivative at fracture intersections on the non conforming mesh.  On each trace $S_m$, $m\in \mathcal{M}$, a mesh is defined and piece-wise constant basis functions are used for the discretization of control variables $\Lambda^m$,  and, independently, another mesh is introduced and piece-wise linear continuous basis functions are used for functions $\Psi^m$. Clearly, different choices for the basis functions of the various variables are possible, the proposed ones being the more natural given the expected regularity of the solution.
It is to remark that the flexibility and robustness in handling non-conforming and independently built discretizations on each fracture and on each trace of the network, for each of the variables involved, is a key aspect of the method, which allows to easily deal with arbitrarily complex geometries without any need of geometrical modification of the DFN. 

The refinement level of the triangular mesh on each fracture is expressed by means of a grid parameter $\bm{\delta_{h}}$, expressing the maximum element area of mesh elements requested on each fracture. Clearly a different grid parameter could be used on each fracture, even if here, for simplicity, a single value is adopted. The refinement level of the meshes on the traces is controlled by two parameters $\bm{\delta_\lambda}$ and $\bm{\delta_\psi}$ representing the ratio between the number of mesh elements on the traces, for $\Lambda$ and $\Psi$ respectively, and the number of elements of the mesh induced by the intersections of the trace with the edges of the triangular mesh. Unique values are used  for $\bm{\delta_\lambda}$ and $\bm{\delta_\psi}$ for all the traces in the network, but different choices are possible.

  \begin{table}
 	\begin{center}
 		\caption{Geometrical details of the considered networks}
 		\label{DFNinfo}
 		\begin{tabular}{|c|c|c|c|c|c|}
 			\hline 
 			 & & & \multicolumn{3}{c|}{Traces per fracture}\\
 			 \cline{4-6}
 			  & Fractures & Traces & average  & min & max\\ 
 			\hline 
 			\textbf{DFN3}  & 3 & 3 &2 &2 &2 \\ 
 			\hline 
 			\textbf{DFN10}  & 10 & 14 & 2.8 & 1 & 5  \\ 
 			\hline 
 			\textbf{DFN395}  & 395 & 629 & 3.18 &1 & 19 \\ 
 			\hline 
 		\end{tabular} 
 	\end{center}
 \end{table}
\subsection{Three fracture DFN problem}
Let us consider the connected domain $\Omega$ shown in Figure~\ref{3fractparaview}, given by the union of three planar fractures defined by
\begin{align*}
&F_1=\left\lbrace (x,y,z) \in \mathbb{R}^3 : -1\leq x\leq 1/2,~-1\leq y \leq 1,~z=0\right\rbrace \\
&F_2=\left\lbrace (x,y,z) \in \mathbb{R}^3 : -1\leq x\leq 0,~y=0,~-1\leq z \leq 1\right\rbrace \\
&F_3=\left\lbrace (x,y,z) \in \mathbb{R}^3 : x=-1/2,~-1\leq y\leq 1,~-1\leq z \leq 1\right\rbrace.
\end{align*}
which intersect forming three traces $S_1=F_1\cap F_2$, $S_2=F_1\cap F_3$ and  $S_3=F_2\cap F_3$. This problem is labeled DFN3. The known hydraulic head distribution $H^{\text{ex}}$ in $\Omega$ is given by
\begin{align}
&H_1^{\text{ex}} (x,y) =\frac{1}{10}\left( - x -\frac{1}{2}\right) \left(8xy(x^2 + y^2)\text{atan2} ( y , x )+ x^3\right) , \label{sol1}\\
&H_2^{\text{ex}} ( x , z ) =\frac{1}{10}\left( - x -\frac{1}{2}\right)x^3 -
\frac{4}{5}\pi\left( - x -\frac{1}{2}\right)
x^3 | z |,\label{sol2}\\
&H_3^{\text{ex}} ( y , z ) = ( y - 1 ) y ( y + 1 )( z - 1 ) z \label{sol3}
\end{align}
being $\text{atan2(y,x)}$ the four quadrant inverse tangent function, and is the solution of the following problem:
\begin{align*}
&-\nabla\cdot\left(\nabla H\right) = -\nabla\cdot\left(\nabla H^{\text{ex}}\right), &\text{in} \ \Omega\setminus \mathcal{S},\\
&H_{|\partial \Omega}=H_{|\partial \Omega}^{\text{ex}}, &\text{on} \ \partial \Omega,
\end{align*}
with additional conditions of continuity and flux conservation at the traces. 
\begin{figure}
	\centering
	\includegraphics[width=0.34\linewidth]{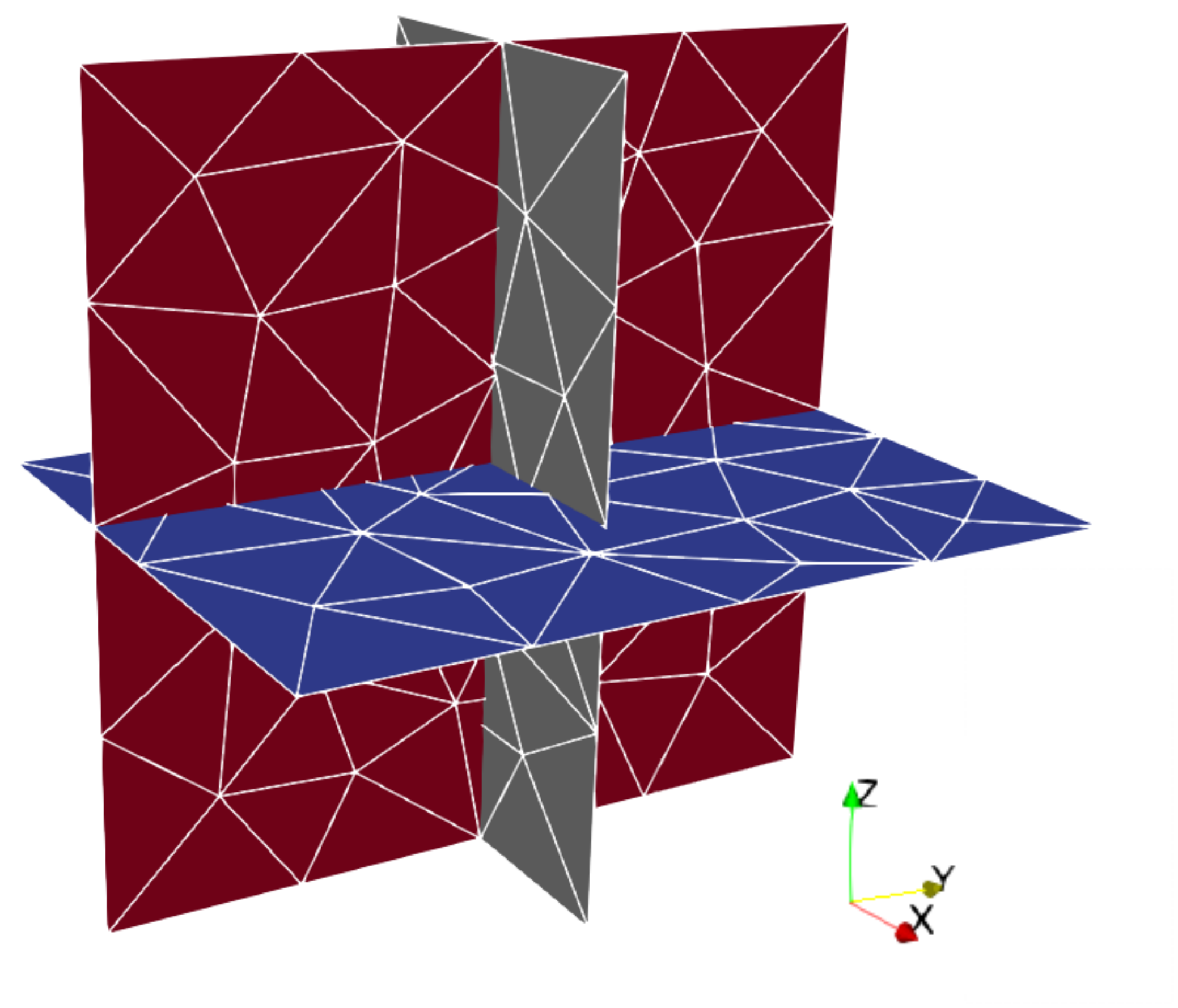}
	\caption{DFN3: DFN configuration.}
	\label{3fractparaview}
\end{figure}
Given the small size of the network, the discrete solution is obtained solving the KKT-system \eqref{KKT}. Five different meshes with an increasing number of elements are considered for the hydraulic head on the fractures, with the mesh parameter $\bm{\delta_{h}}$ ranging between 0.02 and $8\times 10^{-5}$, and nine  values of $\dlam$ and $\dpsi$ are used, both ranging between $0.1$ and $0.9$. The coarsest computational mesh on the fractures is reported in Figure~\ref{3fractparaview}, highlighting the non conformity at fracture intersections. An example solution on the three fractures is reported in Figure~\ref{sol3fract}, for mesh parameters $\bm{\delta_{h}}=0.005$, $\dlam=0.5$ and $\dpsi=0.3$ showing the irregular behavior of the solution across the trace. The use of the XFEM allows to correctly reproduce the jumps of the gradient in the direction normal to the traces even if traces arbitrarily cross mesh elements.  

We computed errors $\mathcal{E}^h_{L^2}$ and $\mathcal{E}^h_{H^1}$ measuring the $L^2(\Omega)$ and $H^1(\Omega)$ norms, respectively, of the relative difference between the numerical and analytical solution for the hydraulic head on the fractures. Error $\mathcal{E}^{\lambda}_{L^2}$ is also computed, expressing the $L^2(\mathcal{S})$ norm of the relative difference between the analytical jump of the fluxes at the traces and the computed value of $\lambda$. The other mesh parameters are fixed with values $\dlam=0.5$ and $\dpsi=0.3$. The behavior of these errors is reported in Figure~\ref{error}: $\EhL$ and $\EhH$ are shown on the left for an increasing number of fracture hydraulic head DOFs and $\Elam$ on the right, for an increasing number of $\lambda$ DOFs on the traces. The expected convergence trend is obtained for $\EhL$ and $\EhH$, despite the non conforming mesh thanks to the use of the XFEM, and the expected convergence trend is obtained also for $\Elam$.

\begin{figure}
	\centering
	\includegraphics[width=43mm]{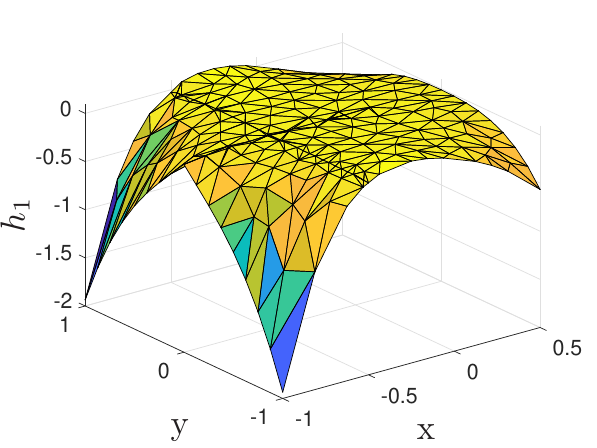}%
	\includegraphics[width=43mm]{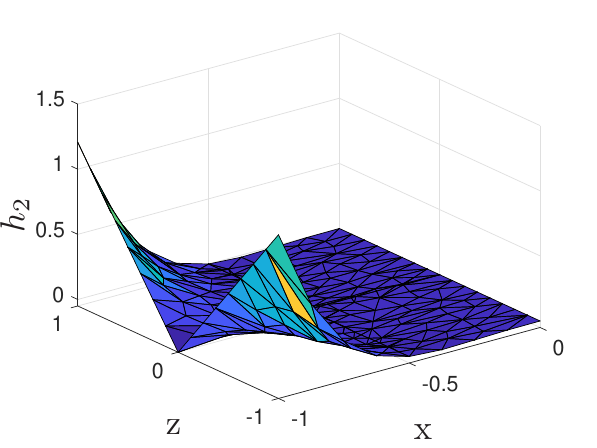}%
	\includegraphics[width=43mm]{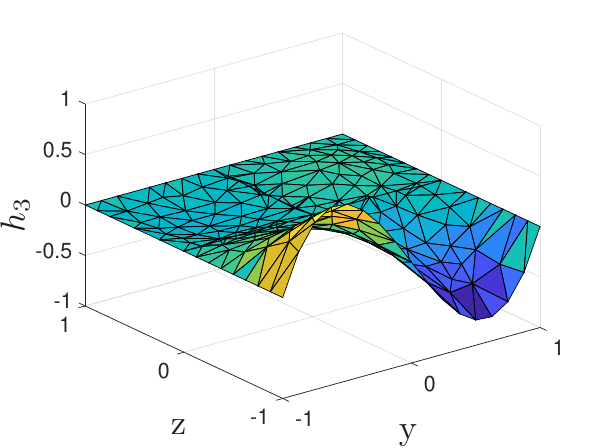}
	\caption{DFN3: hydraulic head computed on the fractures. Parameters: $\dhi=0.0050$, $\dlam=0.5$, $\dpsi=0.3$.}
	\label{sol3fract}
\end{figure}
\begin{figure}
	\centering
	\includegraphics[width=0.5\linewidth]{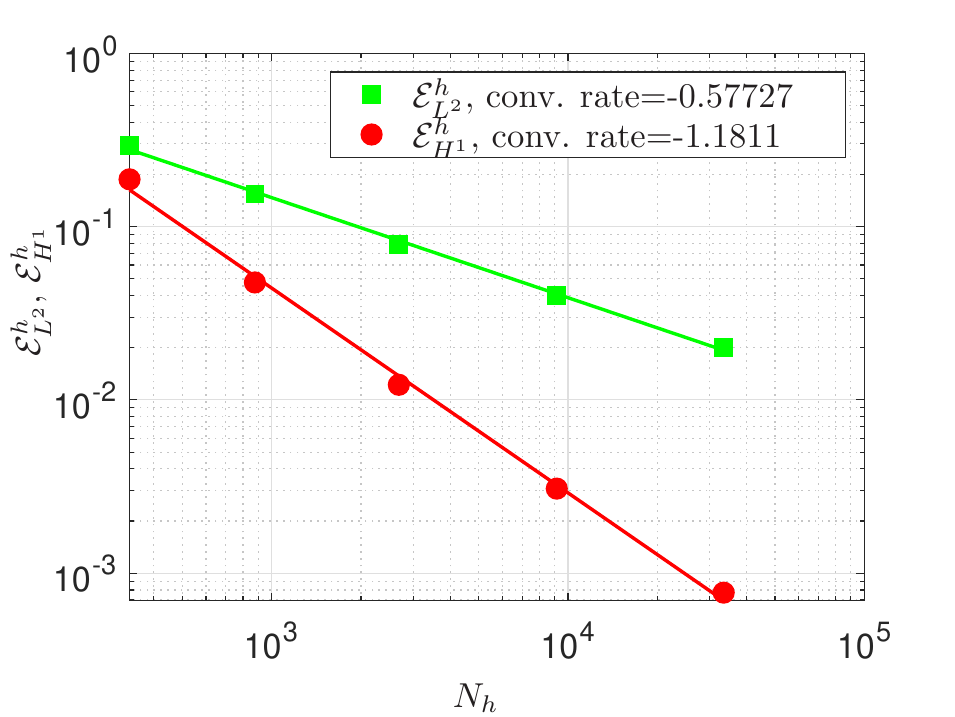}%
	\includegraphics[width=0.5\linewidth]{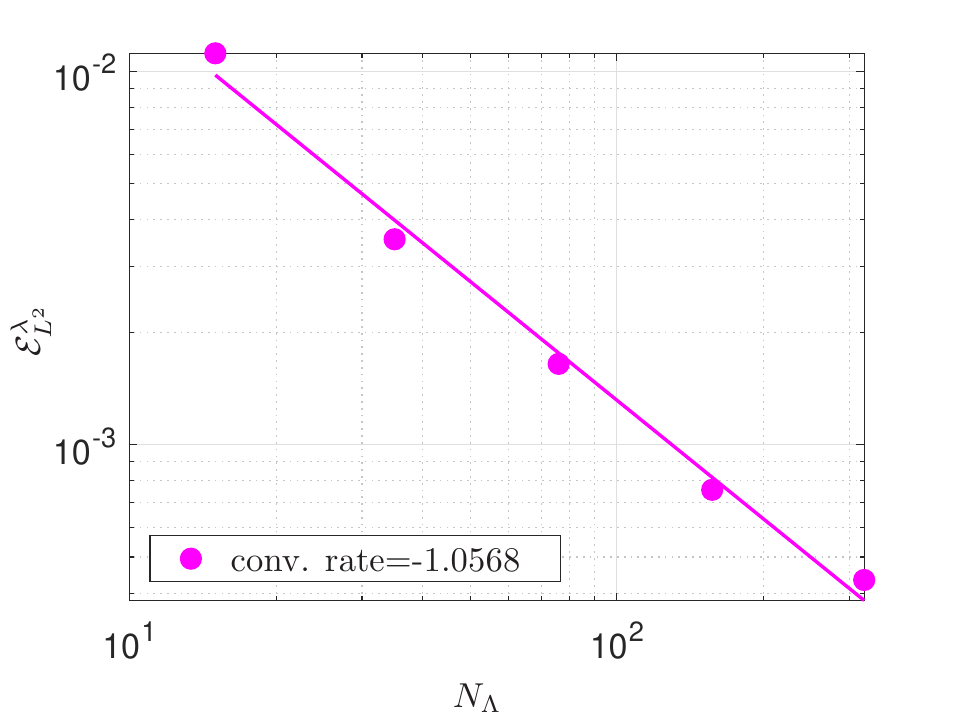}
	\caption{DFN3: on the left, $\mathcal{E}^h_{L^2}$ and $\mathcal{E}^h_{H^1}$ errors under mesh refinement on the fractures; on the right, $\mathcal{E}^{\lambda}_{L^2}$ error under subsequent mesh refinement on the traces. Other Parameters= $\dlam=0.5$, $\dpsi=0.3$.}
	\label{error}
\end{figure}
\begin{figure}
	\centering
	\includegraphics[width=66mm]{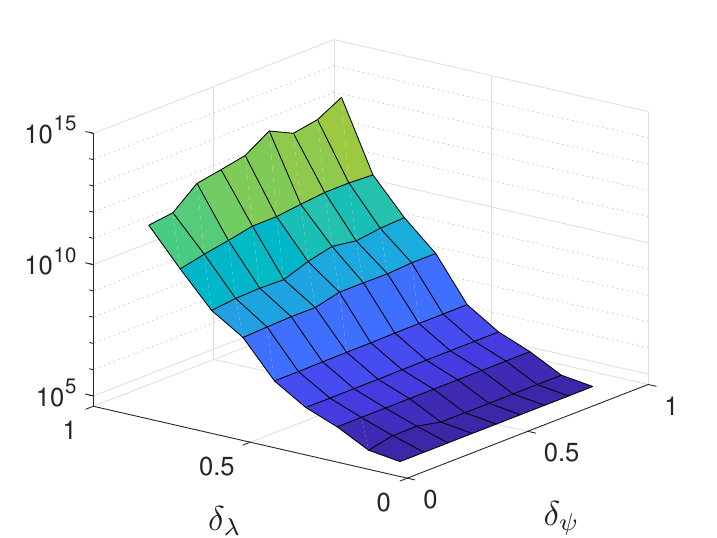}
	\caption{DFN3: Condition number of the matrix $\Mkkt$ varying the parameters $\dlam$ and $\dpsi$. $\dhi=0.0013$.}
	\label{condKKT}
\end{figure}

\begin{figure}
	\centering
	\includegraphics[width=0.45\linewidth]{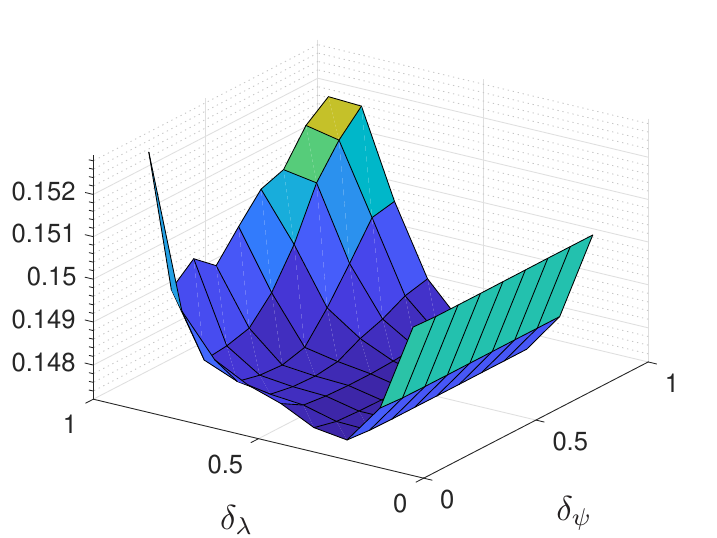}%
	\includegraphics[width=0.45\linewidth]{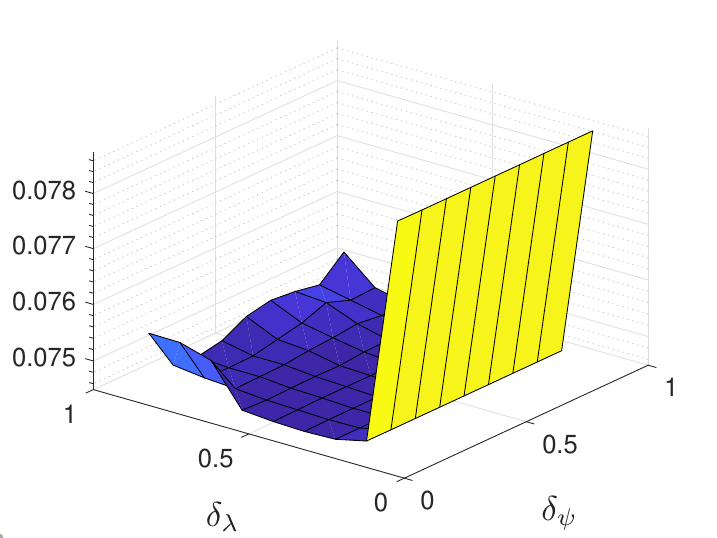}
	\caption{DFN3: Error $\EhH$ varying $\dlam$ and $\dpsi$; $\dhi=0.0050$ on the left and $\dhi=0.0013$ on the right.}
	\label{err_h_200_800}
\end{figure}
\begin{figure}
	\centering
	\includegraphics[width=0.45\linewidth]{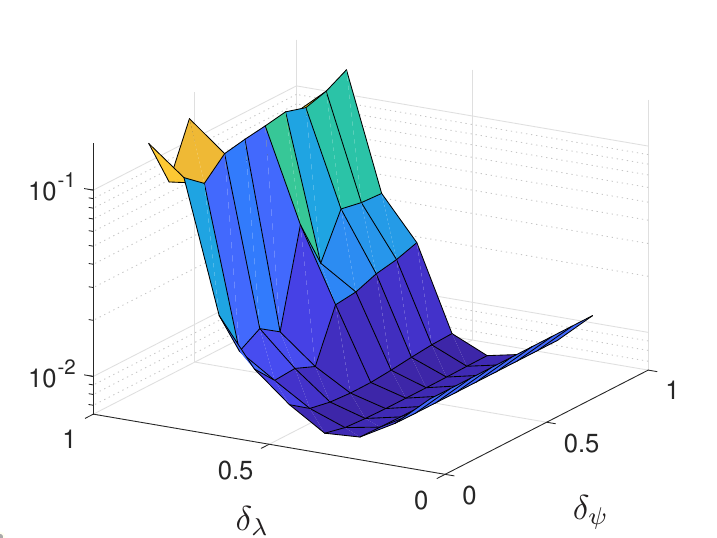}%
	\includegraphics[width=0.45\linewidth]{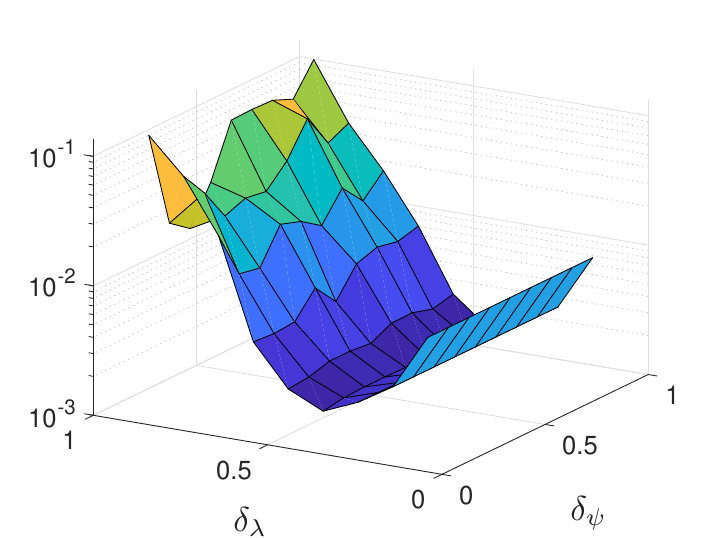}
	\caption{DFN3: Error $\Elam$ varying $\dlam$ and $\dpsi$; $\dhi=0.0050$ on the left and $\dhi=0.0013$ on the right.}
	\label{err_lam_200_800_grad}
\end{figure}

\begin{figure}
\centering
\includegraphics[width=0.45\linewidth]{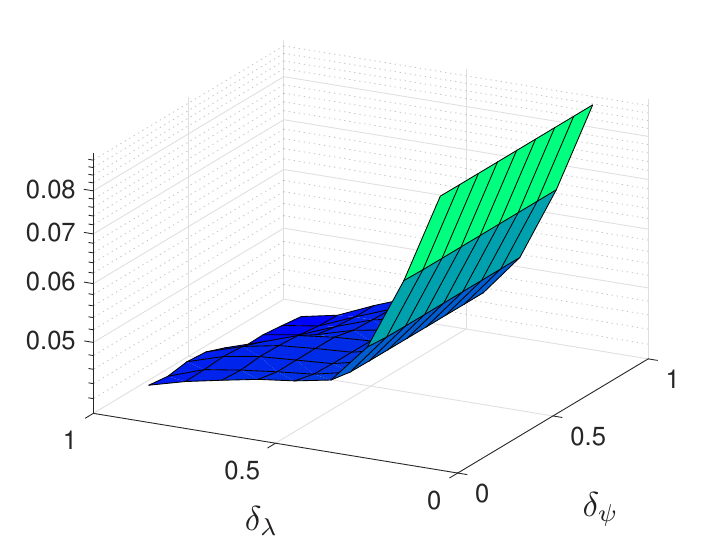}%
\includegraphics[width=0.45\linewidth]{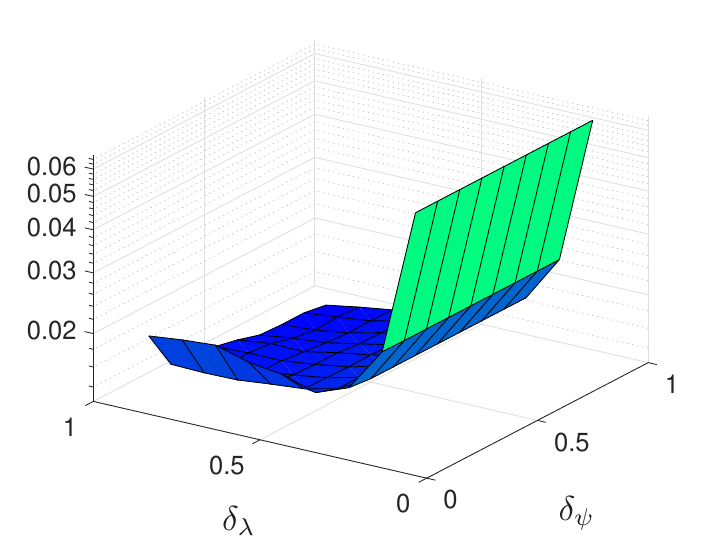}
\caption{DFN3: Error indicator $\Delta_{\mathcal{S}}^h$ varying $\dlam$ and $\dpsi$; $\dhi=0.0050$ on the left and $\dhi=0.0013$ on the right.}
\label{continuitaHH_200_800}
\end{figure}

The effect of the choice of parameters $\dlam$ and $\dpsi$ is also investigated in terms of their influence on the conditioning of the KKT-system and on the accuracy of the solution. Figure~\ref{condKKT} shows the the norm-1 condition number of the KKT matrix for different values of $\dlam$ and $\dpsi$, both ranging between $0.1$ and $0.9$. System conditioning appears to be more affected by parameter $\dlam$, whereas its dependence on $\dpsi$ is almost negligible, especially for the smaller values $\dlam$.
Figures~\ref{err_h_200_800}-\ref{continuitaHH_200_800} show how parameters $\dlam$ and $\dpsi$ impact the quality of the obtained solution for two different values of $\dhi$, $\dhi=0.003$ on the left and $\dhi=7.5\times 10^{-4}$ on the right for all the figures. 
Figure~\ref{err_h_200_800} reports the behavior of error $\EhH$, which appears weakly affected by variations of both the two parameters; a slightly more marked impact of $\dlam$  is observed on the coarsest mesh with a minimum of $\EhH$ for $\dlam$ around $0.5$. This is motivated by the fact that low values of $\dlam$ provide a poor approximation of the flux on the traces which has a detrimental impact on the solution, whereas, when $\dlam$ approaches $0.9$, the solution is affected by the higher conditioning of the system. In Figure~\ref{err_lam_200_800_grad} the trend of error $\Elam$ is described, highlighting, as expected a stronger dependence of this error from $\dlam$, and also an almost no-dependence from $\dpsi$. Again, a minimum of $\Elam$ is observed for values of $\dlam$ arond $0.5$, probably again for the effects of system conditioning at the higher values of this parameter. 
The quantity $\Delta_\mathcal{S}^h$ is now introduced to measure the quality of the hydraulic head solution on the traces, defined as:
\begin{equation}
\Delta_\mathcal{S}^h=\frac{\sqrt{\sum_{m \in \mathcal{M}}\limits||h_i-h_j||^2_{L^2(S_m)}}}{h_{\max} l_{\text{tot}}}, \quad i,j \in I_{S_m},
\label{sommacont}
\end{equation}
being $h_{\max}$ the maximum value of the hydraulic head in $\Omega$ and $l_{\text{tot}}$ the total trace length. Recalling that the continuity of the solution is enforced through the minimization of functional \eqref{Jtilde} by means of the control variable $\psi$, the quantity $\Delta_\mathcal{S}^h$ is an error indicator on the actual continuity achieved by the method across the traces. Local flux conservation is instead intrinsically enforced by the method through the definition of a unique variable for flux jump on the two fractures meeting at each trace. 
In Figure~\ref{continuitaHH_200_800} the behavior of this error indicator is reported. A strong influence of $\dlam$ is again noticed, whereas $\dpsi$ has a minor effect, more evident at high values of $\dlam$. In this case higher values of the parameters provide, in general lower values of $\Delta_\mathcal{S}^h$. Finally, comparing the left and the right pictures of Figures~\ref{err_h_200_800}-\ref{continuitaHH_200_800} we can see that a reduction of the errors and of the error indicator are obtained through a refinement of the mesh.

The effect of conditioning of the KKT-system matrix are actually mitigated by solving the problem via the PCG-solver in Algorithm~\ref{gradcon_prec}, which is actually an application of the null-space method proposed in \cite{Pestana2016} to the saddle-point problem \eqref{KKT}.

\subsection{Ten fracture DFN problem}
A slightly more complex network of $10$ fractures and $14$ traces is now considered, as shown in Figure~\ref{10fractparaview}, and labeled DFN10. The DFN problem is solved on this network using a uniform unitary value of transmissivity for all fractures and with a prescribed unitary head drop between two selected fracture edges, as marked in Figure~\ref{10fractparaview}, and homogeneous Neumann boundary conditions on all other edges. These boundary conditions allow to identify an inflow and an outflow portion of the boundary, as it usually happens in realistic configurations. 
\begin{figure}
	\centering
	\includegraphics[width=0.50\linewidth]{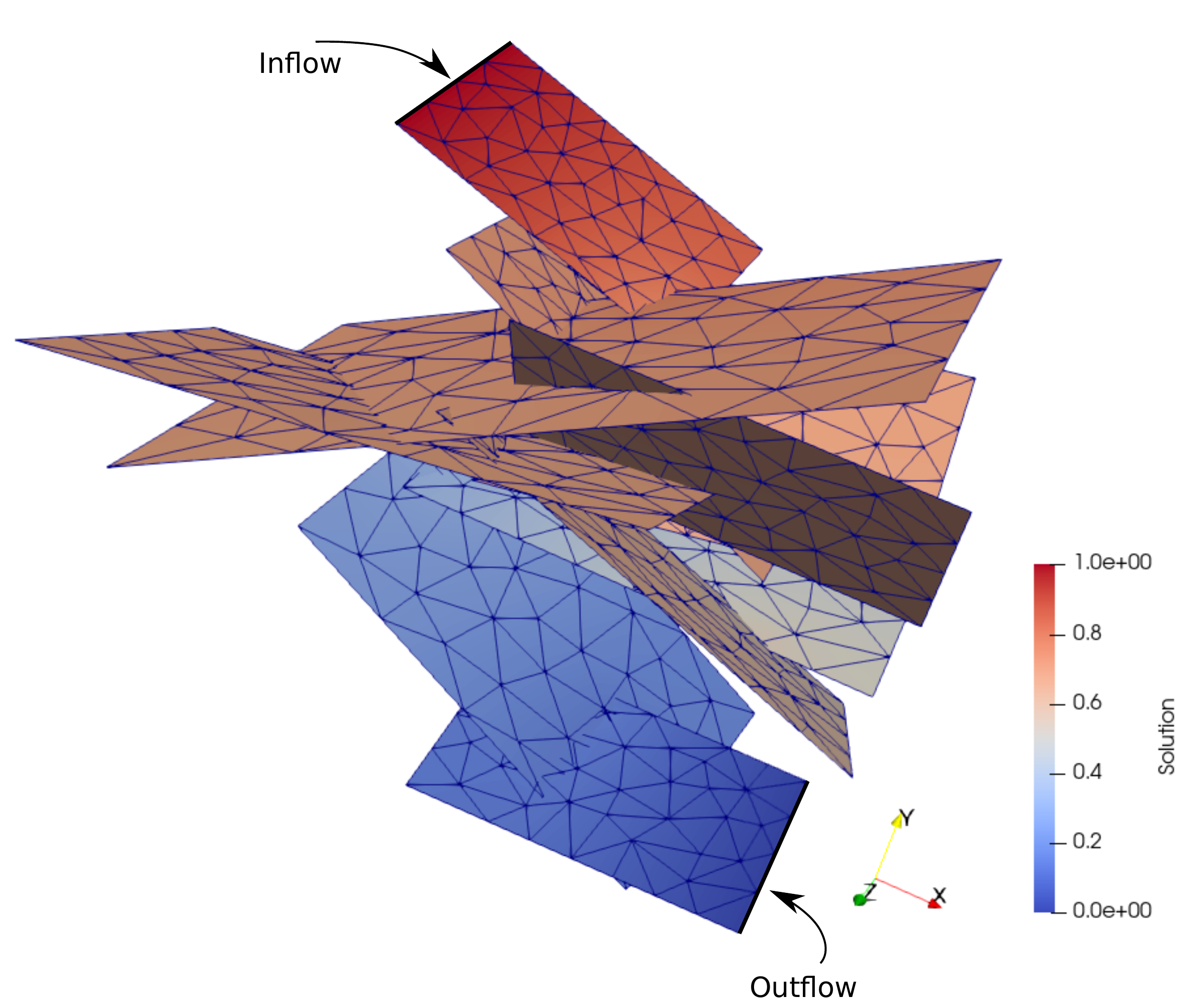}
	\caption{DFN10: DFN configuration}
	\label{10fractparaview}
\end{figure}
An example solution, obtained with the PCG solver, is reported in Figure~\ref{10fractparaview}, along with the non-conforming computational mesh, obtained with $\dhi=0.0011$, $\dlam=0.5$, $\dpsi=0.3$. Figure~\ref{continuitaHH_200_800_DFN10} shows the behavior of the error indicator $\Delta_{\mathcal{S}}^h$ at varying of $\dlam$ and $\dpsi$, on two different meshes, a coarse mesh on the left, with $\dhi=0.0011$ and a fine mesh on the right, with $\dhi=2.7\times 10^{-4}$. As previously noticed, the quantity $\Delta_{\mathcal{S}}^h$ is primarily sensible to variations of parameter $\dlam$, with a decreasing trend for increasing values of $\dlam$. Parameter $\dpsi$ has a minor effect, with a decreasing trend for increasing values of $\dpsi$, more relevant at the higher values of $\dlam$. 

Another error indicator can be introduced, for this configuration, measuring the global flux mismatch between the inflow and the outflow boundary, defined as:
\begin{equation}
\Delta_{\text{in-out}}^\phi = \frac{|\phi_\text{in}-\phi_\text{out}|}{\phi_\text{in}}
\end{equation}
where $\phi_\text{in}$/$\phi_\text{out}$ is the absolute value of the net flux entering/leaving the network through the inflow/outflow boundary. Given the local flux conservation properties of the method at each trace, this quantity is an error indicator of the global conservation properties. The behavior of $\Delta_{\text{in-out}}^\phi$, varying $\dlam$ and $\dpsi$ in the range $[0.1,0.9]^2$, is shown in Figure~\ref{fluxinout200_800}, for two values of $\dhi$, with $\dhi=0.0011$ on the left and $\dhi=2.7\times 10^{-4}$ on the right. It can be seen that the global flux mismatch appears to be affected by variations of $\dlam$, with a generally decreasing trend for increasing values of this parameter, but also a relevant influence from $\dpsi$ appears in this case, mainly at the higher values of $\dlam$, with a decreasing trend for $\Delta_{\text{in-out}}^\phi$ for increasing values of $\dpsi$. The quantity $\Delta_{\text{in-out}}^\phi$  can be reduced also refining the fracture mesh. 

\begin{figure}
	\centering
	\includegraphics[width=0.45\linewidth]{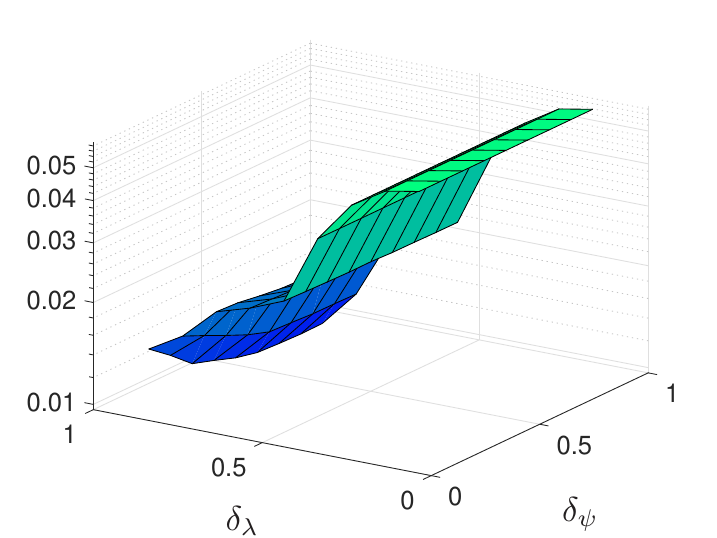}%
	\includegraphics[width=0.45\linewidth]{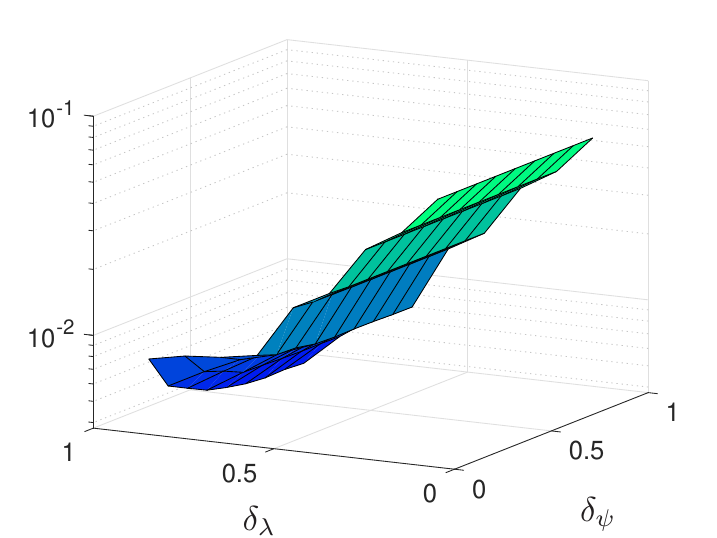}
	\caption{DFN10: Error indicator $\Delta_{\mathcal{S}}^h$ varying $\dlam$ and $\dpsi$; $\dhi=0.0011$ on the left and $\dhi=2.7\times 10^{-4}$ on the right.}
	\label{continuitaHH_200_800_DFN10}
\end{figure}

\begin{figure}
	\centering
	\includegraphics[width=0.45\linewidth]{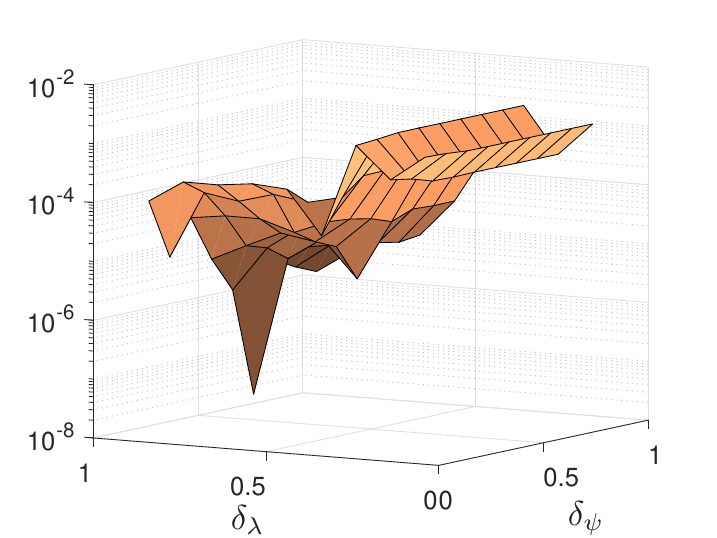}%
	\includegraphics[width=0.45\linewidth]{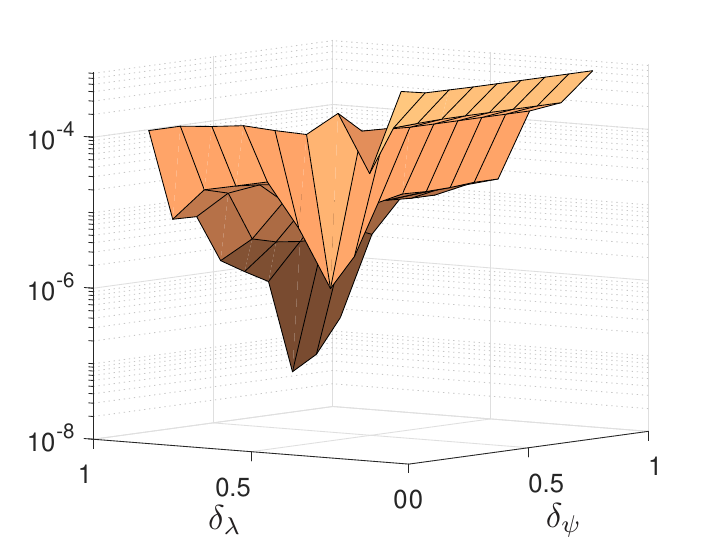}
	\caption{DFN10: Error indicator $\Delta_{\text{in-out}}^\phi$ varying $\dlam$ and $\dpsi$; $\dhi=0.0011$ on the left and $\dhi=2.7\times 10^{-4}$ on the right.}
	\label{fluxinout200_800}
\end{figure}

A study on the performances of preconditioners is proposed on this network. Table~\ref{preconditioners} reports the number of iterations required by the preconditioned conjugate gradient scheme to reduce the relative residual up to $10^{-6}$, for the non-preconditioned case and for preconditioners $\bm{P_f}$ and $\bm{P_d}$ described at the end of Section~\ref{prob_solve}, for four values of $\dhi$ and $\dlam=0.5$, $\dpsi=0.3$. We can see that using preconditioner $\bm{P_f}$, the number of iterations required to reach the required residual is almost unaffected by the value of $\dhi$ and is only about $3.5\%$ of the number of iterations of the non preconditioned case on the finest mesh. The performances of the block-diagonal preconditioner $\bm{P_d}$, suitable for efficient parallel implementation, are slightly worse than the ones relative to preconditioner $\bm{P_f}$, but still only marginally affected by mesh refinement and capable of reducing the iteration to convergence to about $8.6\%$ of the number of iterations of the non preconditioned case on the finest mesh. The sparsity patterns of the full matrix $\hat{\bm{G}}$ in \eqref{matrice} and of $\bm{P_f}$ and $\bm{P_d}$ for mesh parameters $\dhi=0.0011$, $\dlam=0.5$, $\dpsi=0.3$ are shown in Figure~\ref{spy}.

\begin{table}
	\begin{center}
		\caption{DFN10: Number of iterations of PCG algorithm with different preconditioners and mesh refinement; $\dlam=0.5$, $\dpsi=0.3$.}
		\label{preconditioners}
		\begin{tabular}{|c|c|c|c|c|}
			\hline 
			 $\dhi$ &  $\bm{N_{\Lambda}+N_{\psi}}$ & \textbf{non prec.}&  $\bm{P_{f}}$ &  $\bm{P_{d}}$ \\ 
			\hline 
			$797$  & 69 & 147 & 21 & 49 \\ 
			\hline 
			$3120$   & 125 & 234 & 22 & 55  \\ 
			\hline 
			$12428$   & 252 & 457 & 23 & 58 \\ 
			\hline 
			$49784$   & 502 & 717 & 25 & 62\\ 
			\hline
		\end{tabular} 
	\end{center}
\end{table}
\begin{figure}
	\centering
	\includegraphics[width=0.37\textwidth]{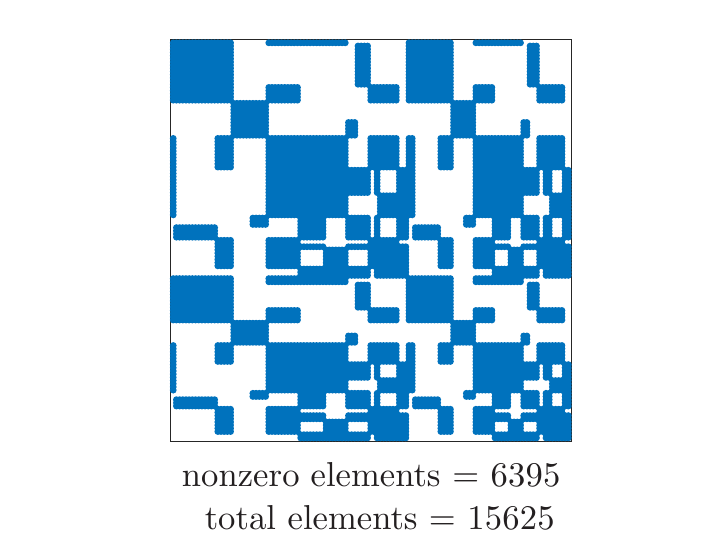}%
	\hspace{-1cm}
	\includegraphics[width=0.37\textwidth]{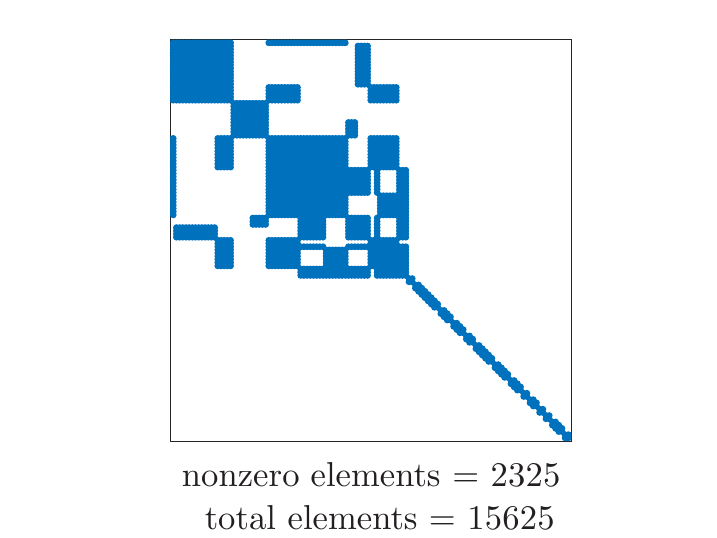}%
	\hspace{-1cm}
	\includegraphics[width=0.37\textwidth]{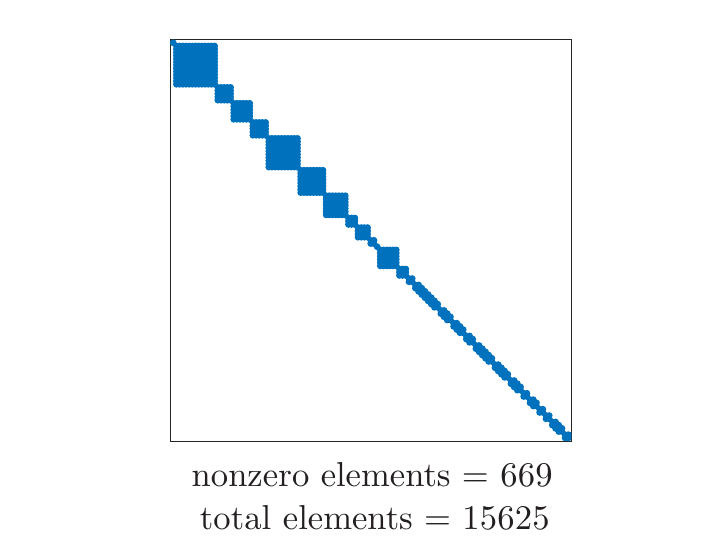}
	\caption{DFN10: Sparsity pattern of $\hat{\bm{G}}$ and of the preconditioners $\bm{P_{f}}$ and $\bm{P_d}$. Parameters: $\dhi=0.0011$, $\dlam=0.5$, $\dpsi=0.3$.}
	\label{spy}
\end{figure}

\subsection{Realistic DFN problem}

\begin{figure}
	\centering
	\includegraphics[width=0.65\linewidth]{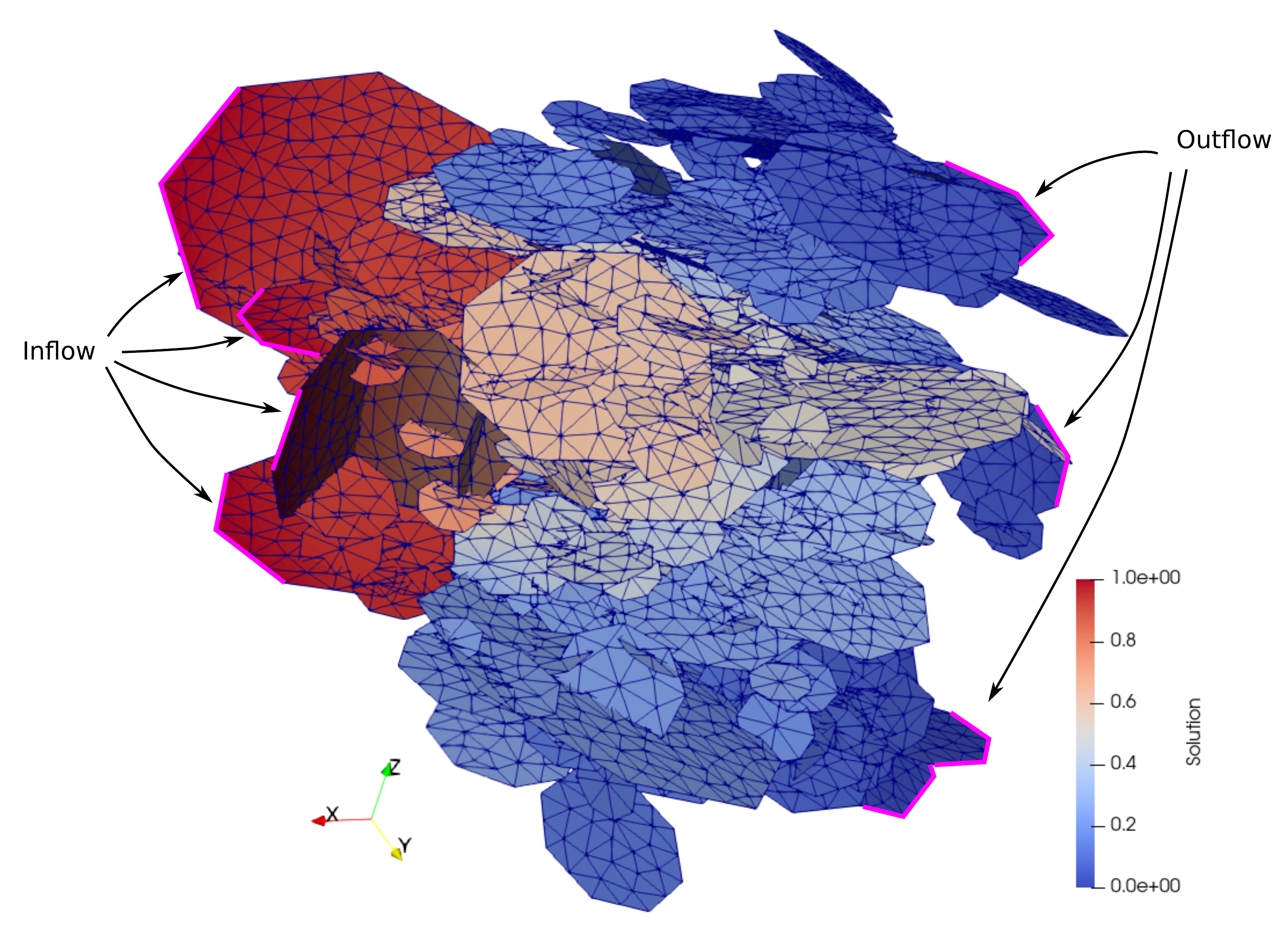}
	\caption{DFN395: Mesh configuration}
	\label{395paraview}
\end{figure}

As a last example, a DFN consisting of $395$ fractures and $629$ traces is considered, labeled DFN395. The DFN is obtained as a realization of probability distribution functions on fracture size, orientation, distribution and hydraulic transmissivity adapted from the data in \cite{DFNData}.
The network is shown in Figure~\ref{395paraview}, along with the inflow and outflow boundary, where Dirichlet  boundary conditions of $1$ and $0$, respectively, are set, all other fracture edges being, instead, insulated. Two simulations are performed with this geometry and boundary conditions: in a first case a uniform transmissivity equal to $\bm{K}=10^{-7}$ is chosen on all fractures, whereas, in a second case, different, constant transmissivity values are used on each fracture, extracted from a log-normal distribution having mean value of the logarithms equal to $\zeta=-5$ and variance $\frac13$.

Let us consider first the case of uniform transmissivity throughout the network: the small value of the transmissivity, compared to the order of magnitude of the hydraulic head, introduces an unbalance among the method's variables, and consequently a re-scaling of the problem is beneficial. This is achieved by introducing a scaling factor $\mathcal{K}$ and redefining the constraint equations of the optimization problem replacing transmissivity $\bm{K}$ by a re-scaled transmissivity $\bm{K}^\star$ given by $K^\star=\mathcal{K}\bm{K}$, thus obtaining a new problem equivalent to the original one in terms of the hydraulic head but having re-scaled fluxes. We refer to \cite{BPSd} for more details on the re-scaling, where this methodology has been proposed in a slightly different context.

Let us solve the re-scaled problem on a mesh with $\dhi=400$, $\dlam=0.5$ and $\dpsi=0.3$ for various values of the scaling factor in the range $10^7<\mathcal{K}<10^{12}$. Figure~\ref{res_iter_risc} shows, on the left, the effect of the scaling on the norm of the initial residual $r_0$ of the PCG method, split into the part relative to $\lambda$, termed $r_0^\lambda=(r_{0,k})_{k=1,\ldots,N_\Lambda}$ and the part relative to $\psi$, $r_0^\psi=(r_{0,k})_{k=N_\Lambda+1,\ldots,N_\Psi}$. Figure~\ref{res_iter_risc} displays instead, on the right, the number of iterations required to solve the problem to a non-preconditioned relative residual of $10^{-6}$, for the non preconditioned case and using preconditioner $\bm{P_f}$ and $\bm{P_d}$, varying $\mathcal{K}$. In Figure~\ref{res_iter_risc}, left, we can see that the initial residual norms become similar, i.e. $\|r_0^\lambda\|\approx \|r_0^\psi\|$, for a value of $\mathcal{K}\approx 10^{9}$. For the same value of $\mathcal{K}$ the number of iterations reaches a minimum as can be seen in Figure~\ref{res_iter_risc}, right. At the minimum, the number of iterations required to solve the problem using the preconditioners is reduced by a factor of about $3$ with respect to the non-preconditioned case, and the performances of preconditioners $\bm{P_f}$ and $\bm{P_d}$ are quite similar. Using preconditioner $\bm{P_f}$, the number of iterations for $\mathcal{K}>10^9$ remains almost fixed, whereas it increases with $\bm{P_d}$, even if of a smaller extent if compared to the non preconditioned case. Values of $\mathcal{K}$ much larger than the optimal should however be avoided as they are expected to increase the conditioning of the problem. 

A rough estimate of the optimal value of $\mathcal{K}$ can be obtained guessing the order of magnitude of the main flux $\phi$ throughout the network. For the present case, given the chosen boundary conditions, flux essentially occurs along the $x$-direction, say $\phi=\phi_x$, whose order of magnitude can be guessed as $\phi_x=\bm{K} \frac{\Delta_x h}{L_x}$, with $\Delta_x h$ equal to the hydraulic head difference along the $x$-direction, and $L_x$ the length of the DFN along $x$, $L_x\approx 1000$, giving $\phi_x\approx 10^{-10}$. As the hydraulic head varies between $1$ and $0$ it is to be expected that a value of $\mathcal{K}$ around $10^{10}$ or slightly less should be used to balance the two terms. 

Similar results are obtained in the case of a different log-normally distributed transmissivities $\bm{K}_i$ among fractures $F_i$, $i\in \mathcal{J}$: in this case the order of magnitude of the flux through the network can be guessed as previously, setting $\bm{K}=10^\zeta$, where $\zeta$ is the mean value of the logarithms of $\bm{K}_i$, $i\in \mathcal{J}$, obtaining  $\phi_x\approx 10^{-8}$. The scaling factor is thus chosen equal to $\mathcal{K} \approx 10^7$ and used for the simulations. Table~\ref{preconditioners395_7} reports the number of iterations required by the PCG solver to reduce the relative residual norm to $10^{-6}$ without preconditioning and with the two preconditioners $\bm{P_f}$ and $\bm{P_d}$, for different values of $\dhi$, ranging between 1600 and 100, $\dlam=0.5$, $\dpsi=0.3$. The values of the two error indicators measuring continuity of the solution and global flux conservation are also reported in the last two columns. We can see that good performances are achieved by the two preconditioners which allow to reduce the number of iterations of a factor up to $5$ for preconditioner $\bm{P_f}$ and up to $3$ with preconditioner $\bm{P_d}$. Both error indicators can be reduced by refining the mesh.

	\begin{figure}
\begin{minipage}{.49\textwidth}
\centering
		\includegraphics[width=0.99\linewidth]{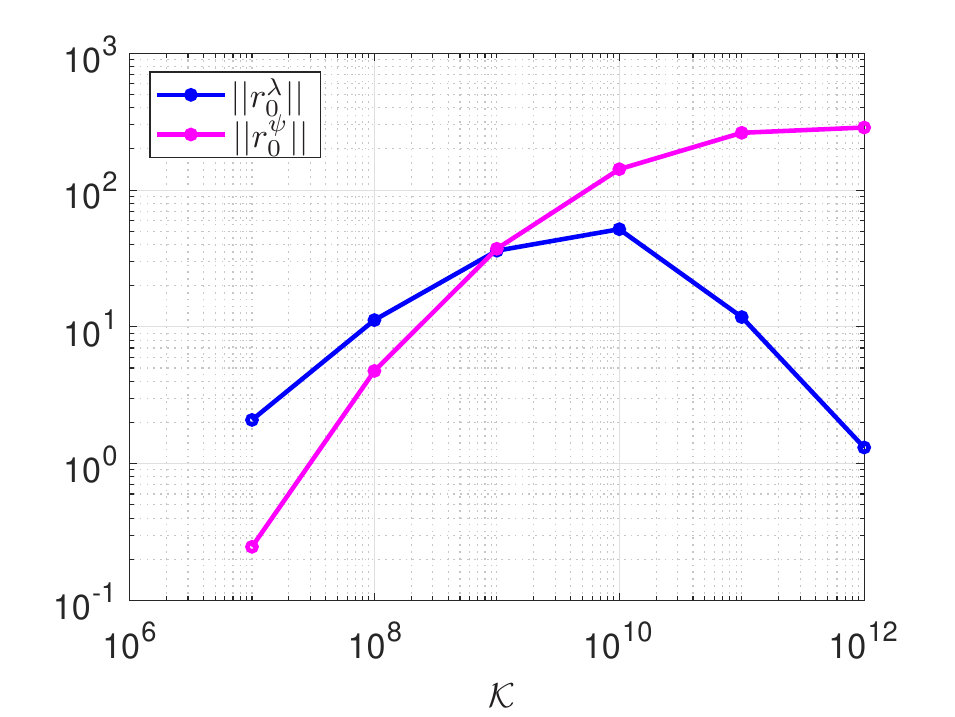}
\end{minipage}
\begin{minipage}[c]{.49\textwidth}
\centering
		\includegraphics[width=0.99\linewidth]{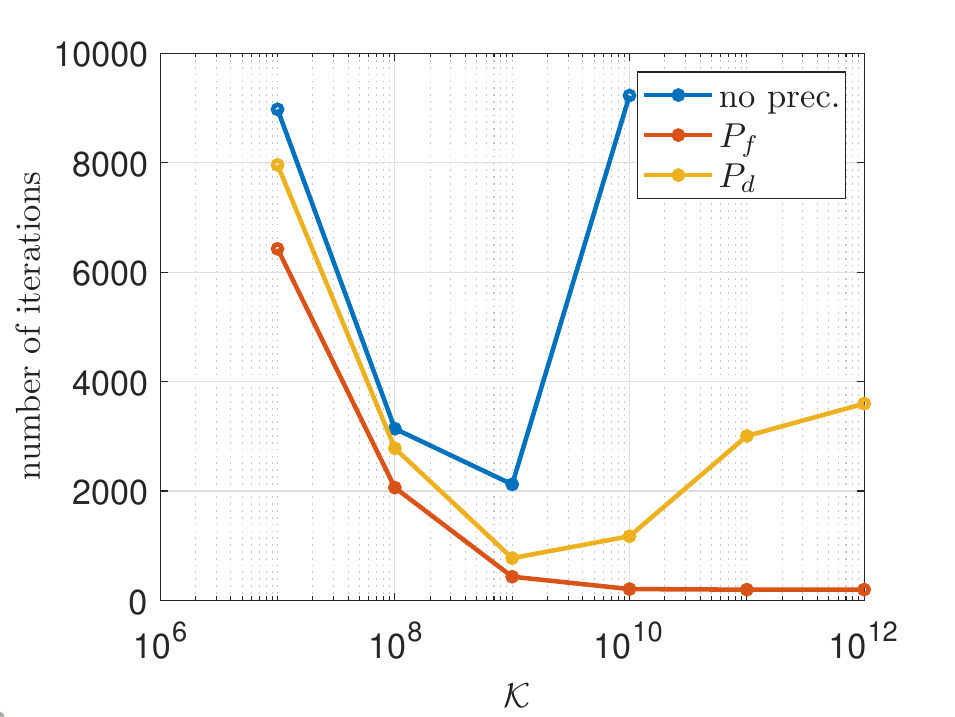}
\end{minipage}
\caption{DFN395: Residual norms $\|r_0^{\lambda}\|$ and $\|r_0^{\psi}\|$ (left) and number of iterations with different preconditioners (right) versus problem scaling. $\dhi=400$ $\dlam=0.5$, $\dpsi=0.3$.}
\label{res_iter_risc}
		\end{figure}

\begin{table}
	\begin{center}
		\caption{DFN395 random transmissivity: number of iterations of PCG algorithm and error indicators under mesh refinement and different preconditioning techniques. $\dlam=0.5$ and $\dpsi=0.3$.}
		\label{preconditioners395_7}
		\begin{tabular}{|c|c|c|c|c|c|c|}
			\cline{3-7} 
			 \multicolumn{1}{c}{~}& \multicolumn{1}{c|}{~} & \multicolumn{3}{c|} {number of iterations} &\multicolumn{2}{c|} {constraints} \\
			\hline
			$\dhi$ &  $N_{\Lambda}+N_{\psi}$ &\textbf{no prec.}&  $\bm{P_{f}}$ &  $\bm{P_{d}}$& $\Delta_\mathcal{S}^h$ &$\Delta^\phi_{\text{in-out}}$\\ 
			\hline 
			$1600$  &2267 & 2322 & 443 &  713 &0.0037  &0.1128\\ 
			\hline 
			$400$   &3606 & 1902 & 486 &  757 &0.0024 & 0.0474 \\ 
			\hline 
			$100$   & 6946 & 1727 & 502 & 847 & 0.0016 &0.0061\\ 
			\hline 
		\end{tabular} 
	\end{center}
\end{table}
\section{Conclusions}
\label{conclusions}
A new approach for flow simulations in geometrically complex fracture networks on non conforming meshes has been formulated and analysed. The method is based on the minimization of a cost functional expressing the error in continuity of the solution at fracture intersection, constrained by PDE equations on the fractures written in a three-field formulation. The resulting discrete problem is well posed independently of any mesh-related aspect, thus ensuring great flexibility to the method in handling arbitrarily complex networks. A solver based on the preconditioned conjugate gradient is designed for the method, ready for implementation on parallel computing architectures. The effect of mesh parameters on the performances of the method have been thoroughly investigated in the numerical example, along with the performances of preconditioning techniques. Local and global flux conservation properties and continuity of the solution at fracture intersections have also been analysed. The method has shown to be effective in solving the flow problem in stochastically generated networks.

\bibliographystyle{siam}
\bibliography{biblio}


\end{document}